\documentclass[11pt]{amsart}

\usepackage[
top = 3.5cm, bottom=3.5cm, left = 2.7cm, right = 2.7cm,
headsep=1.2cm
]{geometry}

\usepackage[LGR, T1]{fontenc}

\usepackage{centernot}
\usepackage{amssymb, amsmath, amsthm}
\usepackage{mathrsfs}
\usepackage{comment}
\usepackage{cleveref}
\usepackage{graphicx} %scalebox
\usepackage{cjhebrew}

\usepackage{xcolor}

\newcommand{\todo}[1]{\textcolor{red}{#1}}

\theoremstyle{plain}
\newtheorem{theorem}{Theorem}
\newtheorem{lemma}[theorem]{Lemma}
\newtheorem{proposition}[theorem]{Proposition}
\newtheorem{corollary}[theorem]{Corollary}
\theoremstyle{definition}
\newtheorem{definition}[theorem]{Definition}
\theoremstyle{remark}
\newtheorem{remark}{Remark}

\usepackage{cleveref}

\crefname{lemma}{lemma}{lemmas}
\Crefname{lemma}{Lemma}{Lemmas}
\crefname{proposition}{proposition}{propositions}
\crefname{corollary}{corollary}{corollaries}
\crefname{definition}{definition}{definitions}

\newcommand{\rlzin}{\mathop{\varepsilon}}
\newcommand{\notrlzin}{\centernot{\varepsilon}}

\newcommand{\Perp}{\mathop{\rlap{$\bot$}\mkern3mu \raisebox{0ex}{$\bot$}}}
\newcommand{\scalePerp}{\scalebox{0.75}{$\boldsymbol{\Perp}$}}
\newcommand{\falsity}[1]{|| #1 ||}
\newcommand{\falsitydual}[1]{\falsity{#1}^{\scalePerp}}
\newcommand{\dual}[1]{{#1}^{\scalePerp}}
\DeclareRobustCommand{\qproofs}{%
  \mathord{\text{\usefont{LGR}{STEP-TLF}{m}{n}\symbol{"15}}}%
}

\newcommand{\closed}{fully-closed\,}
\newcommand{\closure}{full-closure\,}

\newcommand{\dsucc}{\mathop{\rlap{$\succ$}\mkern4mu \raisebox{0ex}{$\succ$}}}
\newcommand{\len}{\ell}

\newcommand{\cc}{\mathsf{cc}}
\newcommand{\BR}{\mathbf{B}}
\newcommand{\Restr}{{\bf Restr}}
\newcommand{\Ind}{{\bf Rec}}

\newcommand{\dfn}{\mathrel{::=}}

\newcommand{\op}{\textnormal{\texttt{op}}}

\newcommand{\imp}{\rightarrow}
\newcommand{\subterm}[1]{\mathrm{SubT}({#1})}
\newcommand{\substacks}[1]{\mathrm{SubS}({#1})}
\newcommand{\FV}[1]{\mathrm{FV}({#1})}

\newcommand{\dom}[1]{\mathrm{dom}({#1})}
\newcommand{\im}[1]{\mathrm{im}(#1)}

\newcommand{\restr}{\upharpoonright}

\newcommand{\fullname}[1]{\text{\Large \cjRL{r}}({#1})}

\title{Bar-recursion and preservation of cardinals}

\author{Laura Fontanella}
\address{Univ Paris Est Creteil, LACL, [F-94010 Creteil, France]}
\email{laura.fontanella@u-pec.fr}
\author{Jacopo Furlan}
\address{Laboratoire d'Informatique de Paris Nord, Université Sorbonne Paris Nord\\ 
and Univ Paris Est Creteil, LACL, [F-94010 Creteil, France]}
\email{furlan@lipn.univ-paris13.fr}
\begin{document}

\begin{abstract} %Building up from some work of J-L. Krivine \cite{DBLP:conf/csl/Krivine16}, 

This work presents a transfinite version of the bar-recursion in the context of classical realizability models for set theory. Bar-recursion has been previously used to obtain realizability interpretations of countable choice and dependent choice, and was employed by Krivine to realize the continuum hypothesis in classical realizability. In this paper, we introduce a transfinite variant of bar-recursion and use it to construct realizability models validating uncountable fragments of the Axiom of Choice. Moreover, our construction reveals that this generalized bar-recursion is related to preservation of cardinals. To show this, we define an analogue of the forcing notion of $\kappa$-closure for classical realizability algebras that we call $\kappa$-\closed. We show that, in realizability algebras satisfying the $\kappa$-\closure property, generalized bar-recursion realizes that any cardinal up to $\kappa$ admits a representative in the realizability model which remains a cardinal.
\end{abstract}

\maketitle


\noindent {\bf Acknowledgements:} This research was partially supported by the \emph{Agence National de le Recherche} (ANR), project Cocontens. 

\section{Introduction}

Bar-induction, originally formulated by Brouwer \cite{brouwer1927definition} in intuitionistic mathematics, plays a crucial role in understanding well-foundedness properties and constructive interpretations of choice principles. It has been studied in computer science from proof-theoretic, semantic, and operational perspectives. The literature on applications of bar-induction in computer science is vast, we may cite a few results in this field. Bernardi et al. \cite{BR:DBLP:conf/esop/BernardiCLS25}  
used bar-induction for characterizing properties of orders on parallel processes. Rahli et al. \cite{DBLP:journals/jacm/Liron} have shown that bar-induction is compatible with constructive type
theory, thus providing a theoretical framework for using bar-induction in constructive work and integrating it into proof-assistants such as Rocq or Agda.
Brede and Herbelin \cite{DBLP:conf/lics/BredeH21} studied the duality between bar-induction and choice principles.

Bar-induction yields a higher-type recursion scheme known as \emph{bar-recursion}, introduced by Spector \cite{Spector1962} who used it to extend Gödel's functional interpretation of Peano Arithmetic to Classical Analysis via a double-negation translation of Countable Choice, CC. 
%provide a functional interpretation of classical analysis intended as First-Order Arithmetic plus Dependent Choice. 
Subsequent
work has refined and extended this work, showing how suitable variants of bar
recursion provide computational content for CC and Dependent Choice, DC, in
different logical settings. Berardi, Bezem and Coquand \cite{DBLP:journals/jsyml/BerardiBC98}
used a form of bar-recursion for a realizability interpretation of CC and DC in classical logic.   
Berger and Oliva \cite{Berger_Oliva_2005} proposed a modified bar-recursion for realizing CC and DC;  the result is a refinement of Berardi, Bezem and Coquand's work without using their infinitary calculus. Blot and Riba \cite{DBLP:conf/aplas/BlotR13} showed how Berger and Oliva's modified bar-recursion can be used to realize CC in Parigot's $\lambda\mu$-calculus.
Streicher \cite{DBLP:journals/lmcs/Streicher17} developed a realizability model arising from a model of untyped lambda calculus in coherence spaces and used Berger and Oliva's bar-recursion to realize DC in this model. 

%\todo{Mancano BerardiOS19 (indispensabile?) et Escardo+Oliva14 (indispensabile?); forse anche Oliva+Powell 2018?}

%constructive proof of termination
%BR:DBLP:journals/logcom/BerardiOS19

%general BI 
%BR:DBLP:journals/corr/EscardoO14

Starting from these results, Krivine \cite{DBLP:conf/csl/Krivine16} adapted bar-recursion to the formalism of classical realizability, obtaining a realizer for CC and DC in a realizability interpretation of \emph{classical set theory}. 
This paper generalizes Krivine's 
version of bar recursion to a form of transfinite (possibly uncountable) recursion 
%over infinite sequences of terms 
and combines it with some well-known properties of forcing to realize,  not only countable, but also \emph{uncountable} fragments of the Axiom of Choice, AC, in a classical setting.
%transfinite case, relating the realizability approach with some well-known topics of forcing theory.

Classical realizability was introduced by Krivine in \cite{DBLP:journals/aml/Krivine01}, and developed in a series of papers \cite{DBLP:conf/types/Krivine14,DBLP:conf/csl/Krivine16, DBLP:journals/mscs/Krivine18, DBLP:journals/lmcs/Krivine21} that 
extend the realizability tradition started with Kleene \cite{DBLP:journals/jsyml/Kleene45} to classical theories, in particular to classical set theory. This work was inspired by the discovery of a proof-as-programs correspondence for classical logic according to 
Griffin's work \cite{DBLP:conf/popl/Griffin90}.
Krivine's technique 
is based on algebraic structures called \emph{realizability algebras}, which 
extend combinatory algebras
to provide a computational interpretation of classical logic. 
Realizability algebras are based on an orthogonality relation between programs and stacks, defining the so-called Krivine's abstract machine, KAM. In this setting, programs are formalized in
the $\lambda_\mathsf{c}$-calculus, a variant of the $\lambda$-calculus introduced in Felleisen et al. \cite{DBLP:conf/lics/FelleisenFKD86} which corresponds to an idealized Scheme programming language with a term for \texttt{call/cc}.
It is worth pointing out that the architecture of the KAM has found interesting applications in computer science beyond the scope of realizability interpretations of set theory, we mention some results: Diaz-Caro et al. \cite{diazcaro:hal-02175168} presented a semantics for a simply typed linear algebraic lambda-calculus based
on the KAM architecture, and showed that this provides a
framework to express \emph{quantum} control (i. e. operations on quantum data); inspired by the KAM, Accattoli et al. \cite{DBLP:conf/lics/AccattoliLV21} introduced the Space KAM, a space cost model for the $\lambda$-calculus accounting for logarithmic space; the KAM was also used by Ehrhard and Regnier \cite{DBLP:conf/cie/EhrhardR06} to show that the normal form of the Taylor expansion of a $\lambda$-term equals the Taylor expansion of its Böhm tree. 

%AccattoliVanoniDalLagoLICS2021 The Space of Interactions: introduced Space KAM, inspired by KAM, a space cost model for the lambda calculus accounting for logarithmic  space 
%(De Carvalo PhD thesis dimostra che a partire dal sistema dei multitypes - modello LI- se il modello ha un certo tipo ha un bound corrispondente al numero di passi di riduzione di testa linear  -linear reduction steps- modellizato dalla KAM, questo ci da un costo in termini di tempo non di spazio)

Classical realizability generalizes the method of Forcing in set theory, since every boolean-valued model can be naturally interpreted as a classical realizability model (we refer to \cite{DBLP:journals/Richard} for the details of this translation). Nevertheless, from a computational point of view, forcing models are not very informative, since all the realizers are interpreted as the same element (the top element of the boolean algebra), that is why forcing models are called \emph{trivial} realizability models. 
Certain set theoretic principles beyond the scope of $\mathrm{ZF}$ are hard to interpret in non-trivial classical realizability ---in particular one of the most challenging problems is how to realize the full Axiom of Choice. Krivine was able to realize Countable and Dependent choice using a version of bar-recursion, or quote, or a clock. Later Fontanella and Geoffroy \cite{FontanellaG20} realized uncountable fragments of Zorn's Lemma 
using an operator $\chi$ that we call here \emph{trichotomy operator} combined with an additional property called the \emph{$\kappa$-chain condition} (generalizing the homonymous property of forcing). 
%\todo{citare Fontanella Geoffroy}\\
In \cite{DBLP:journals/lmcs/Krivine21}, Krivine has shown that the full $\mathrm{AC}$ can hold in a (non-trivial) realizability model, however it is not clear from his construction how to extract an explicit realizer for $\mathrm{AC}$; thus, its computational interpretation remains unsettled. 
%In this paper we introduce a transfinite version of bar-induction and the corresponding bar-recursion operator, and show how this can be used to realize uncountable fragments of choice, proving that bar-recursion is linked to choice principles even at uncountable levels.  

The study of bar principles highlighted bar recursion as a central mechanism linking inductive principles on trees over finite sequences of natural numbers, with choice principles in both constructive and classical settings. 
This paper explores a generalized form of bar-induction for trees over \emph{infinite} (possibly uncountable) sequences of ordinals (Definition \ref{def: transfinite bar induction}), and introduces an operator that we call \emph{transfinite bar-recursion} (Definition \ref{def: transfinite bar recursion}) for implementing a form of bar-recursion over infinite (possibly uncountable) sequences of terms. In this paper, we show that even at uncountable levels, bar-recursion and choice principles are entangled.

This work requires a careful treatment of ordinals and cardinals in classical realizability models. Indeed, realizability models for set theory are built starting from ground ZF models that are used to name the elements of the resulting realizability model, however an uncountable cardinal in the ground model, may `collapse' to a countable ordinal in the ground model. Since our focus is on \emph{uncountable} fragments of the axiom of choice, this phenomenon could interfere with the result and lead us to realize merely countable choice. 
In the case of forcing constructions, several properties ensure preservation of cardinals, one of them is the so-called $\kappa$-\emph{closure} ensuring that all cardinals up to $\kappa$ are preserved. In this paper, we introduce an analogous property for classical realizability algebras that we call  \emph{$\kappa$-\closure} (Definition \ref{def: kappa-closure}). $\kappa$-\closure plays a role similar to the continuity axiom in Baire space in the proof of termination of bar-recursion. We make the link with forcing explicit: for realizability algebras arising from a boolean algebra $\mathbb{B}$, the $\kappa$-\closure coincides with $\mathbb{B}^+$ being $\kappa$-closed according to the usual set theoretic definition.

We show that in realizability algebras that satisfy the $\kappa$-\closure where $\kappa$ is the size of the algebra, our transfinite bar recursion realizes fragments of the Axiom of Choice below $\kappa$ (Theorem \ref{thm:ACk}). Moreover, 
we show that in classical realizability algebras, the $\kappa$-\closure can be combined with our transfinite bar-recursion to preserve cardinals up to $\kappa$ (Theorem \ref{thm:preserving cardinals}). Thus uncountable fragments of choice can be realized starting from uncountable realizability algebras. 

The main contributions of this paper are the following:
\begin{itemize}
    \item we define the notion of $\kappa$-\closure to realizability algebras, showing that it corresponds to a generalization of the $\kappa$-closure principle in forcing case;
    \item the introduction of a variant of bar-induction for trees of infinite (possibly uncountable) sequences, which is shown to be a theorem of $\mathrm{ZFC}$;
    \item we define a transfinite bar-recursion operator corresponding to the computational counterpart of our transfinite bar-induction and we give a proof of termination by means of transfinite bar-induction in the context of a $\kappa$-\closed realizability algebra;
    \item we prove that our transfinite bar recursion combined with the $\kappa$-\closure yields 
    a realizability model which satisfies the Axiom of Choice for families indexed below a representative of $\kappa,$ denoted $\hat{\kappa},$ and ensure the preservation of cardinals up to $\hat{\kappa}$ (i.e. the representative $\hat{\mu}$ of any cardinal $\mu\leq\kappa$ remains a cardinal in this model). 
\end{itemize}

The paper is structured as follows: in \cref{sec: classical realizability} and section \ref{subsec: names} we recall the general technique for building classical realizability models for set theory; in \cref{sec:closure} we introduce the $\kappa$-\closure property in classical realizability; in \cref{subsec: exemple} we illustrate an example of (non-trivial) realizability algebra satisfying this property; \cref{sec: bar induction} introduces a transfinite bar induction principle and the corresponding transfinite bar-recursion operator. 
In \cref{sec: ACkappa} and \cref{sec: cardinal preservation}, we show that transfinite bar-recursion combined with $\kappa$-\closure can be used to realize $\mathrm{AC}_{<\hat{\kappa}}$ (the Axiom of Choice for families index by ordinals below $\kappa$), and to preserve cardinal up to $\hat{\kappa}.$   

\begin{comment}
Countable Axiom of Choice and Dependent Choice can be realized using Krivine's version of the bar-recursion \cite{DBLP:conf/csl/Krivine16}, while uncountable fragments of the Axiom of Choice have been realized by Fontanella and Geoffroy in \cite{FontanellaG20} using an operator $\chi$ that we call here \emph{trichotomy operator} combined with the assumption of an additional property for realizability algebras known as the \emph{$\kappa$-chain condition} (inspired by a well-known property of forcing). While Krivine's intepretation with bar recursion is limited to countable fragments of the axiom of choice, Fontanella and Geoffroy's work uses an operator (the trichotomy operator) that has no natural interpretation in computational models. In this paper we combine these constructions and present a different realizability interpretation for countable as well as uncountable fragments of the Axiom of Choice using a generalization of Krivine's bar recursion combined with a property that we call $\kappa$-closure corresponding to a well-known property of forcing having the same name. 
\end{comment}

\section{Classical Realizability algebras}\label{sec: classical realizability}

In this section we recall the basic notions of classical realizability, in particular we 
present \emph{realizability algebras} which are the main ingredient for the construction of classical realizability models for set theory. Realizability algebras can be thought of as generalizations of the notion of boolean algebras, that involve programs and stacks interacting as players and opponents of formulas to build a coherent theory. 

A realizability algebra is a tuple $\mathcal{K} = \langle\Lambda_\mathsf{c}, \Pi, \succ, \Perp\rangle,$ where $\Lambda_\mathsf{c}$ is a set of \emph{programs} or \emph{terms}, $\Pi$ is the set of \emph{stacks}, $\succ$ defines a notion of execution in some abstract machine, $\Perp$ is the so-called \emph{pole} and defines a notion of orthogonality between programs and stacks. These notions can be customized, but classical realizability imposes certain minimum requirements presented hereafter. 

$\Lambda_\mathsf{c}$ should be a variant of the $\lambda_\mathsf{c}$-calculus, 
that is the set of terms $t$ generated by the following grammar for a fixed set of variables $V$ and a (possibly empty) set of constants $A$:
    \[ t\;\dfn\; x\mid a\mid\cc\mid\lambda x.t\mid
    (t)t\mid k_\pi\quad\text{for $x\in V,a\in A, \pi\in\Pi$}.\] 

When $A$ is non-empty, the constants in $A$ will be called \emph{special instructions}. The term $\cc$ is called \emph{call-cc} and is needed for interpreting classical logic, constants $k_\pi$ are called \emph{continuations} of $\pi.$ We can add additional instructions to this calculus depending on which properties we want to be satisfied in the corresponding realizability model, in particular we are allowed to deal with an uncountable set of terms. Parenthesis are considered as left associative, i. e. $tu_1\dots u_n = (\cdots((t)u_1)\cdots)u_n$. $\lambda$-abstractions are written in contracted form, so that $\lambda x_1\cdots x_n.t$ denotes $\lambda x_1.\cdots\lambda x_n.t$\,. The set of free variables of a term $t\in\Lambda_\mathsf{c}$ is denoted by $\FV{t}$. The set of closed $\lambda_\mathsf{c}$-terms is denoted by $\Lambda_\mathsf{c}^{\textrm{closed}}$, namely terms where every variable is under the scope of a $\lambda$-abstraction; likewise, $\Lambda^{\textrm{closed}}$ denotes the set of closed $\lambda$-terms. The definition of substitution for $\lambda_\mathsf{c}$-terms corresponds to the usual one in $\lambda$-calculus and it is denoted by curly brackets. The term $\cc$ and stack constants $k_\pi$ are considered as constants, that means $\cc\{x:=u\} = \cc, k_\pi\{x:=u\}=k_\pi$ for each variable $x$ and each term $u$.

\emph{Stacks} are sequences of terms in $\Lambda_\mathsf{c}^{\textrm{closed}}$. We shall fix a set of \emph{stack bottoms} (also called \emph{atomic stacks}) denoted $\Pi_0$, then $\Pi$ is defined as $[\Lambda_\mathsf{c}^{\textrm{closed}}]^{<\omega}\times\Pi_0$, or equivalently it is generated by the following grammar: 
	\[\pi\;\dfn\;\pi_0\mid t\centerdot\pi\quad\textrm{ where }\pi_0\in\Pi_0, t\in\Lambda_\mathsf{c}^{\textrm{closed}}.\]
The set of \emph{processes} $\Lambda_\mathsf{c}\star\Pi$ is defined as $\Lambda_\mathsf{c}^{\textrm{closed}}\times\Pi$. A process $p=(t, \pi)\in\Lambda_\mathsf{c}\star\Pi$ will be denoted by $t\star\pi$. 

The \emph{reduction} $\succ$ describes an abstract machine (the KAM), and it is a relation on $\Lambda_\mathsf{c}^{\textrm{closed}}\star\Pi$ that can be customized but should satisfy at least the following rules: for any $t, u\in \Lambda_\mathsf{c}^{\textrm{closed}}, \pi,\rho\in\Pi$,
\begin{align*}
    \lambda x.t\star u\centerdot\pi&\succ_{grab} t\{x:=u\}\star\pi\tag{\emph{grab}},\\
    (t)u\star\pi&\succ_{push} t\star u\centerdot\pi\tag{\emph{push}},\\
    \cc\star t\centerdot\pi&\succ_{save} t\star k_\pi\centerdot\pi\tag{\emph{save}},\\
    k_\pi\star t\centerdot\rho&\succ_{restore} t\star\pi\tag{\emph{restore}}.
\end{align*}
The labels of the reduction are omitted unless relevant for the discussion. Constants can be evaluated by adding further rules (see e.g. Definition \ref{def:trichotomyoperator}). The reflexive and transitive closure of the relation $\succ$ is denoted by the symbol $\dsucc$. Given two terms $t, u\in \lambda_c$ we write $t\succ u$ when $t\star \pi\succ u\star \pi$ for all $\pi\in \Pi$. Similarly, $t\dsucc u$ denotes that either $t=u$ or there exist terms $t_0,\dots t_n$ such that $t_0=t, t_n=u$ and $t_i\succ t_{i+1}$ for $0\leq i<n$.

The \emph{pole} $\Perp\subseteq \Lambda_\mathsf{c}\star\Pi$ is a subset of processes closed by antireduction, i.e. if $u\star\rho\succ t\star\pi$ and $t\star\pi\in\Perp$, then $u\star\rho\in\Perp$; in particular, the push rule ensures that if $t\star u\centerdot\pi$ is in the pole, so is $(t)u\star\pi$. 

This concludes the definition of a realizability algebra. All over the paper, $\mathcal{K}$ will denote some specific or arbitrary realizability algebra. 

\begin{remark}
    In this paper we will also assume that if $(t)u\star \pi\in \Perp$ then $t\star u\centerdot \pi\in \Perp,$ and if $\lambda x.t\star u\centerdot\pi\in \Perp$ then $ t\{x:=u\}\star\pi\in \Perp$.
\end{remark}

Given a set of stacks, the pole allows for naturally associating to every set of stacks, a set of orthogonal terms. 

 \begin{definition}
    For any set of stacks $X\subseteq\Pi$, the set of orthogonal terms is defined by 
        \[\dual{X}\dfn\{t\in\Lambda_\mathsf{c}\mid\forall \pi\in X(t\star\pi\in\Perp)\}.\] 
\end{definition}

In order to use a realizability algebra to build models of a classical theory, one assigns to each formula $F$ in the language $\mathscr{L}$ of the theory, a \emph{falsity value}  denoted $\falsity{F},$ and a 
\emph{truth value} denoted $\falsitydual{F}$. Truth and falsity values are defined simultaneously according to the notion of orthogonality induced by the pole, and we assume that formulas are written using  $\top,$ $\bot,$ $\imp$ and $\forall$ as primitive logical symbols and the other logical operators are defined from these ones as follows:
    \[
    \begin{array}{ll}
    \lnot F \equiv F\imp \bot,& 
    F\land G\equiv (F\rightarrow G\rightarrow\bot)\rightarrow\bot,\\    
    F\vee G\equiv (F\rightarrow\bot)\rightarrow (G\rightarrow\bot)\rightarrow\bot, &
    \exists xF\equiv(\forall x(F\rightarrow\bot))\rightarrow\bot.
    \end{array}
    \]
In the following, we will denote by $A_1,\dots, A_n\rightarrow B$ formulas in the form $A_1\rightarrow\dots\rightarrow A_n\rightarrow B$, the latter being logically equivalent to $A_1\land\dots\land A_n\rightarrow B$.

Truth values are defined by 
$$\falsitydual{F}:=\{t\in \Lambda_\mathsf{c}\mid\forall \pi\in \falsity{F}(t\star \pi\in \Perp)\}.$$
Thus, the truth value of a formula is the orthogonal of its falsity value.

Falsity values must satisfy the following conditions: 
\begin{itemize}
\item $\falsity{\top}=\emptyset$ and $\falsity{\bot}=\Pi$; 
\item $\falsity{A\imp B}=\{t\centerdot \pi;\ t\in\falsitydual{A}\textrm{ and }\pi\in \falsity{B}\}$
\end{itemize}
In subsection \ref{subsec: names} we will complete the definition of falsity values for the formulas of set theory.

We need to select special terms so that, when we consider the formulas whose truth value contains one of those terms, we have a coherent theory; traditionally such designated terms are called \emph{proof-like terms}.  

\begin{definition}\label{def: quasi-proofs}
    A set of \emph{proof-like terms}, denoted $\qproofs$, is a subset of $\Lambda_\mathsf{c}$ such that 
    \begin{itemize}
    \item $\qproofs$ is closed by application 
    \item $\cc\in\qproofs$ and $\Lambda^{\textrm{closed}}\subset\qproofs$ (where $\Lambda^{\textrm{closed}}$ denotes the set of all closed terms of pure $\lambda$-calculus)
    \item $\qproofs\cap\falsitydual{\bot}=\emptyset.$
    \end{itemize}
    Given $t\in \qproofs$ and a formula $F$, we say that $t$ \emph{realizes} $F$ and write $t\Vdash F$ when $t\in \falsitydual{F}.$ 
    We write $\Vdash F$ when there is $t\in \qproofs$ such that $t\Vdash F.$  
    The theory $\{F\in\mathscr{F}_\mathscr{L}\mid\ \Vdash F\}$ is the \emph{realizability theory} induced by a fixed pole $\Perp$ for the language $\mathscr{L}$.
\end{definition}

The so-called \emph{adequacy lemma} (see e.g. \cite{DBLP:journals/aml/Krivine01}) ensures that the realizability theory contains classical tautologies and it is closed by modus ponens. 

In Definition \ref{def: quasi-proofs}, the condition $\forall t\in\qproofs\,\exists \pi\in\Pi(t\star\pi\not\in\Perp)$ is required to ensure that proof-like terms won't realize contradictory statements, so that the realizability theory is consistent and has a model.

\begin{remark}
    Observe that if $\Perp\neq\emptyset$, then $\qproofs\neq\Lambda_\mathsf{c}$. Indeed, if $t\star\pi\in\Perp$, then for any $\rho\in\Pi$ we have $(k_\pi)t\star\rho\succ t\star\pi$ which implies $(k_\pi)t\star\rho\in\Perp$, thus $(k_\pi)t\not\in\qproofs$.
\end{remark}

By abuse of terminology the expression \emph{the realizability model} refers to any arbitrary model of the realizability theory; traditionally we also denote by $\mathcal{N}$ any such model. 

We conclude this section by fixing the notation for some technical tools needed in the following. 
\emph{Contexts} can be defined by the grammar
    \[C\dfn [\;]_i\mid\lambda x.C\mid Ct\mid tC\mid CC\ \textrm{ for any variable $x$, for any $\lambda_\mathsf{c}$-term $t$ and for $i\in\omega$.}\]
This notion of context corresponds to \emph{multiple numbered contexts} in \cite{DBLP:books/daglib/0067558}. We will use the compact notation $C[t_1,\dots, t_n]$ to refer to a context where the holes are (orderwise) replaced by terms $t_1,\dots, t_n$. For example, if $C[\;]=\lambda x.[\;]_{2}(x[\;]_1)$, then $C[t,u]$ is the term $\lambda x.u(xt)$. We will avoid the use of indexes when it is not necessary.

We also generalize the definition of subterm to $\lambda_\mathsf{c}$-terms and stacks. In addition to the usual definition, the cases below are required: for each term $t$, and for each stack bottom $\pi_0$ and any stack $\pi,$
    \[ \subterm{k_\pi}\dfn\subterm{\pi},\quad
    \subterm{\pi_0}\dfn\emptyset,\quad \subterm{t\centerdot\pi}\dfn\subterm{t}\cup\subterm{\pi}.
    \]
The notion of substack of a $\lambda_\mathsf{c}$-term or of stacks is defined by
\[\substacks{x}\dfn\substacks{a}\dfn\emptyset,\quad\substacks{\lambda x.t}\dfn\substacks{t},\quad\substacks{(t)u}\dfn\substacks{t}\cup\substacks{u},\]
\[\substacks{k_\pi}\dfn\substacks{\pi},\quad\substacks{\pi_0}\dfn\{\pi_0\},\quad\substacks{t\centerdot\pi}=\{t\centerdot\pi\}\cup\substacks{t}\cup\substacks{\pi},\]
for $x\in V, a\in A\cup\{\cc\}, t, u\in\Lambda_\mathsf{c}, \pi_0\in\Pi_0, \pi\in\Pi$.

Given constants $a\in A$, the substitution of $a$ by $u$ in a term $t$ or in a stack $\pi$ is defined by\\[6pt]
\begin{tabular}{rcl}
$a\{a:=u\}$ &$=$& $u;$\\
$a'\{a:=u\}$ &$=$ &$a'\quad(a'\neq a, a'\in A);$\\
$x\{a:=u\}$ &$=$& $x;$\\
$(\lambda x.t')\{a:=u\}$&$=$&$\lambda x.t'\{a:=u\};$
\end{tabular}
\begin{tabular}{rcl}
$((t_1)t_2)\{a:=u\}$&$=$&$(t_1\{a:=u\}) t_2\{a:=u\};$\\
$k_\varsigma\{a:=u\}$&$=$&$k_{\varsigma\{a:=u\}}\quad(\varsigma\in\Pi);$\\
$\sigma\{a:=u\}$&$=$&$\sigma\quad(\sigma\in\Pi_0);$\\
$(t'\centerdot\pi)\{a:=u\}$&$=$&$t'\{a:=u\}\centerdot \pi\{a:=u\}.$
\end{tabular}\\[4pt]
The definition naturally extends to processes: for $p= t\star\pi, p\{a:=u\}\dfn (t\{a:=u\})\star(\pi\{a:=u\})$

Given stack bottoms $\sigma\in \Pi_0$, the substitution of $\sigma$ by $\rho$ in a term $t$ or in a stack $\pi$ is defined by\\[6pt]
\begin{tabular}{rcl}
$a\{\sigma:=\rho\}$&$=$&$a;$\\
$x\{\sigma:=\rho\}$&$=$&$x;$\\
$(\lambda x.t')\{\sigma:=\rho\}$&$=$&$\lambda x.t'\{\sigma:=\rho\};$\\
$((t_1)t_2)\{\sigma:=\rho\}$&$=$&$(t_1\{\sigma:=\rho\}) t_2\{\sigma:=\rho\};$\\
\end{tabular}
\begin{tabular}{rcl}
$k_\varsigma\{\sigma:=\rho\}$&$=$&$k_{\varsigma\{\sigma:=\rho\}}\quad(\varsigma\in\Pi);$\\
$\sigma\{\sigma:=\rho\}$&$=$&$\rho;$\\
$\sigma'\{\sigma:=\rho\}$&$=$&$\sigma'\quad(\sigma'\neq\sigma, \sigma'\in\Pi_0);$\\
$(t'\centerdot\pi)\{\sigma:=\rho\}$&$=$&$t'\{\sigma:=\rho\}\centerdot \pi\{\sigma:=\rho\}.$
\end{tabular}

The definition naturally extends to processes: for $p= t\star\pi, p\{\sigma:=\rho\}\dfn (t\{\sigma:=\rho\})\star(\pi\{\sigma:=\rho\})$.
%%%

\section{Realizability models for set theory}\label{subsec: names}

We work with a classical first-order theory without equality. 
This is due to the ambiguous nature of equality which may be intended as an extensional relation 
(two sets are equal when they have the same elements) or as Leibniz identity 
(two sets are identical when they satisfy the same formulas of the language) 
which is not necessarily an extensional relation. 
Our framework will be sensible to such a difference as we will work with a version of set theory, 
called $\mathrm{ZF}_\varepsilon,$ where two equalities can be defined, an extensional one ($\simeq$) and a non-extensional one ($=$).
The realizability models we discuss are consequently models of $\mathrm{ZF}_{\rlzin}$, 
but they induce a model $\mathrm{ZF}$ since the former is a conservative extension of the latter.

The theory $\mathrm{ZF}_\varepsilon$ first appeared in Friedman \cite{DBLP:journals/jsyml/Friedman73} then re-elaborated by Krivine \cite{DBLP:journals/aml/Krivine01} in order to develop a realizability interpretation for classical set theory. 
%elaborated Friedman's work to develop a consistent realizability interpretation for classical set theory. 
%The theory examined is called $\mathrm{ZF}_\varepsilon$, 
The language of $\mathrm{ZF}_\varepsilon$ is $\mathscr{L}_{\varepsilon}\dfn\{\notrlzin, \not\in, \subseteq\}$, 
and the set of formulas in this language is denoted by $\mathscr{F}_\varepsilon$. 
For a comprehensive presentation of $\mathrm{ZF}_\varepsilon$ we refer to \cite{DBLP:journals/corr/abs-1007-0825, DBLP:journals/Richard}; 
for the present paper, it is sufficient to point out that $\mathrm{ZF}_\varepsilon$ encompasses two different membership notions $\rlzin$ and $\in$, 
respectively non-extensional and extensional. For technical reasons we take as primitives their negative forms $\notrlzin$ and $\notin$. 
The relationship $x\subseteq y$ is defined by the formula $\forall z(z\rlzin x\rightarrow z\in y)$ 
and leads to the definition of the extensional equality $x\simeq y$ as $(x\subseteq y)\land(y\subseteq x)$.

We fix a ground model $\mathcal{M}$ of $\mathrm{ZFC}$ and a realizability algebra $\mathcal{K}\in\mathcal{M}$. In order to specify the falsity value of each formula, it is necessary to fix the support for interpreting open formulas. This is done by defining \emph{names} within $\mathcal{M}$.

\begin{definition} 
For every ordinal $\alpha$, we inductively define  a set $M^{\mathcal{K}}_\alpha$ as follows : 
$$M^{\mathcal{K}}_\alpha\dfn\bigcup_{\beta<\alpha} \mathcal{P}(M_\beta\times\Pi).$$
Then the class of \emph{names}, denoted $M^{\mathcal{K}},$ is defined by 
$M^{\mathcal{K}}:=\bigcup_{\alpha\in Ord} M^{\mathcal{K}}_\alpha.$
\end{definition}

\begin{definition} 
    Given formula $F\in\mathscr{F}_\varepsilon$, we define its \emph{falsity value}, denoted $\falsity{F},$ by induction on the structure of $F$ as follows:
    \begin{itemize}
        \item $\falsity{\top}\dfn\emptyset, \falsity{\bot}\dfn\Pi$;
        \item  For every $c,d\in M^{\mathcal{K}},$
        \begin{itemize} 
            \item $\falsity{c\notrlzin d} \dfn \{\pi\in\Pi\mid \langle c,\pi\rangle\in d\}$
            \item $||c\not\in d|| \dfn \bigcup_{e\in M^{\mathcal{K}}} \{t\centerdot t'\centerdot\pi\in\Pi\mid\langle e,\pi\rangle\in d, t\Vdash e\subseteq c, t'\Vdash c\subseteq e\}$;
            \item $||c\subseteq d|| \dfn \bigcup_{e\in M^\mathcal{K}} \{t\centerdot \pi\in\Pi\mid\langle e,\pi\rangle\in c, t\Vdash e\not\in d\}$
             \end{itemize}
        (this is well defined by induction on the rank of the names $c, d$).     
        \item $\falsity{F_1\rightarrow F_2}\dfn\{t\centerdot\pi\in\Pi\mid t\in\falsitydual{F_1}, \pi\in\falsity{F_2}\}$;
        \item $\falsity{\forall xG}\dfn\bigcup_{c\in M^{\mathcal{K}}} \falsity{G[c/x]}$.
    \end{itemize}
\end{definition}

\begin{theorem}[\cite{DBLP:journals/aml/Krivine01}] 
The axioms of $\mathrm{ZF}_\varepsilon$ are realized.
\end{theorem}

\subsection{Equalities}
Along with the extensional equality defined above, a different definition of equality is introduced by specifying its falsity values. For any $c,d\in M^{\mathcal{K}}$, the falsity value of the formula $c=d$ is defined as
    \[
        \falsity{c=d}\dfn
        \begin{cases} \emptyset & M\models c=d,\\ \Pi & M\models c\neq d, \end{cases}
    \]
which is, $\Vdash c=d$ if and only if $c$ and $d$ are the same element in the ground model. It can be shown that $\Vdash\forall x\forall y(x=y\leftrightarrow\forall z(x\rlzin z\rightarrow y\rlzin z))$, which means that $=$ corresponds to Leibniz equality.  

\subsection{Canonical representatives}

\bigskip

The ground model $\mathcal{M}$ will be embedded in the realizability model using two operators, the so-called \emph{gimel operator} denoted $\gimel$ and the \emph{reish operator} denoted $\text{\Large \cjRL{r}}.$

\begin{definition}
Given a set of names $c\subset\mathcal{M}^\mathcal{K}$,  we define  $\gimel(c):= c \times \Pi\in\mathcal{M}^\mathcal{K}$.
For $x\in\mathcal{M}$, the canonical representative of $x$ in the realizability model, denoted $\fullname{x}$, is inductively defined by
    \[\fullname{x}\dfn\{\langle \fullname{y},\pi\rangle\mid y\in x, \pi\in\Pi\},\]
or equivalently, $\fullname x=\gimel\{\fullname{y}\mid y\in x\}$. Names of the form $\fullname{x}$ are called \emph{full names}. 
\end{definition}
The canonicity of full names is justified by the fact that for any element $y$ of the ground model, $y\in x$ if and only if $\Vdash \fullname{y}\rlzin\fullname{x}$.

Despite their canonicity, full names may be interpreted differently from their original counterpart. For instance, $\fullname{2}$ contains more than two non-extensional elements in non-trivial realizability models (see e. g. \cite{DBLP:journals/corr/abs-1007-0825, DBLP:conf/types/Krivine14, DBLP:journals/mscs/Krivine18, DBLP:journals/Richard, DBLP:journals/mscs/Miquel20}). It can be shown that when $\fullname{2}$ has only two non-extensional element, the $\mathrm{ZF}$-model induced by the realizabilty algebra is isomorphic to a Forcing model (see \cite[Proposition 3.30 and Remark 3.31]{DBLP:journals/mscs/Miquel20}), which are models where every statement is realized by a unique term.

The gimel operator and the reish operator allow for the definition of restricted quantifier: given a formula $F(z)$, we define 
    \[\falsity{\forall z^{\gimel{c}}\,F(z)}\dfn\bigcup_{d\in c}\falsity{F[d]},\]
    \[\falsity{\forall z^{\fullname{x}}\,F(z)}\dfn\bigcup_{y\in x}\falsity{F[\fullname{y}]},\]
for all $c\subseteq \mathcal{M}^\mathcal{K}, x\in\mathcal{M}$, which mimic restricted quantification since the formulas 
    \[\forall z\rlzin {\gimel(c)}\, F(z)\leftrightarrow\forall z^{\gimel(c)}\,F(z),\quad\forall z\rlzin {\fullname{x}}\, F(z)\leftrightarrow\forall z^{\fullname{x}}\,F(z),\] 
are realized by a term which does not depend on $c, x$ nor $F$ (see for instance \cite{FontanellaM25}). 

\subsection{Ordinals in realizability models}

We discuss here the notion of ordinals in $\mathrm{ZF}_{\rlzin}$, starting by adapting some definitions of order theory to our framework.

\begin{definition}
    In $\mathrm{ZF}_{\rlzin}$, a set $a$ is 
    \begin{itemize}
    \item \emph{$\rlzin$-transitive} if for $x\rlzin a$, $y\rlzin x$ implies $y\rlzin a$;
    \item \emph{$\rlzin$-totally-ordered} (or $\rlzin$-TOD) if for all $x,y\rlzin a$ such that $x\neq y$, either $x\rlzin y$ or $y\rlzin x$ holds; 
    \item \emph{$\rlzin$-well-ordered} if it is $\rlzin$-TOD and for any $\rlzin$-subset of $a$ there exists a least element w.r.t. the $\rlzin$ relation. 
    \end{itemize}
By substituting $\rlzin$ by $\in$ (and $\neq$ by $\not\simeq$) in previous definition, we obtain the usual set theoretical notions of $\in$-transitive set, $\in$-totally-ordered set and $\in$-well-ordered set. 
\end{definition}

In order to establish a workable definition of ordinals in $\mathrm{ZF}_{\rlzin}$, consider the following modifications from standard set-theoretical definitions:
\begin{enumerate}
\item $x$ is a $\rlzin$-transitive set which is $\rlzin$-well-ordered;
\item $x$ is a $\rlzin$-transitive set which is $\rlzin$-TOD;
\item $x$ is a $\rlzin$-transitive set of $\rlzin$-transitive sets.
\end{enumerate}
From an extensional point of view (i.e. considering $\in$ instead of $\rlzin$ in theabove), the three definitions are equivalent and yield the usual notion of ordinal. The same doesn't hold for the $\rlzin$ relation, for $(1)$ and $(2)$ are equivalent yet stronger than $(3)$.
Furthermore, as shown in \cite{FontanellaM25}, it can be realized that the class of all $\rlzin$-transitive sets which are $\rlzin$-TOD is a set, thus the notion of $\rlzin$-TOD does not capture the whole class of $\in$-ordinals in the $\mathrm{ZF}$-model induced by the realizability model. This discussion motivates the following definition, while the subsequent proposition confirms that a $\rlzin$-ordinal is an ordinal for the extensional membership.

\begin{definition}
A $\rlzin$-ordinal is an $\rlzin$-transitive set of $\rlzin$-transitive sets. 
\end{definition}

\begin{proposition}[\cite{FontanellaM25}]
Every $\rlzin$-ordinal is an $\in$-ordinal (i.e. an $\in$-transitive set of $\in$-transitive sets). 
\end{proposition}

\begin{comment}
From an extensional point of view, the three definitions yield the same notion of ordinal. Indeed, in $\mathrm{ZF}$, an ordinal can be equivalently defined as 
\begin{enumerate}
\item a $\in$-transitive set which is $\in$-well-ordered;
\item a $\in$-transitive set which is $\in$-TOD;
\item a $\in$-transitive set of $\in$-transitive sets.
\end{enumerate}
The same doesn't hold for the $\rlzin$ relation. In $\mathrm{ZF}_{\rlzin},$ replacing $\in$ with $\rlzin$ in $(1)-(3)$, these notions are no longer equivalent: $(1)$ and $(2)$ are equivalent yet stronger than $(3)$. Furthermore, as shown in \cite{FontanellaM25}, it can be realized that the class of all $\rlzin$-transitive sets which are $\rlzin$-TOD is a set, thus the notion of $\rlzin$-TOD does not capture the whole class of $\in$-ordinals in the $\mathrm{ZF}$-model induced by the realizability model. A more appropriate translation of the notion of ordinal in the language of  $\mathrm{ZF}_{\rlzin}$ is the one of an $\rlzin$-transitive set of $\rlzin$-sets. Indeed, an $\rlzin$-ordinal in this sense is also an $\in$-ordinal (intended as an $\in$-transitive set of $\in$-transitive sets) as proven in \cite{FontanellaM25}. Moreover, we shall see hereafter that, using the reish operator, we can define a class of $\rlzin$-ordinals that is extensionally equivalent to the class of $\in$-ordinals. 
\end{comment}

We shall briefly discuss here two classes of names representing ordinals. One uses the reish operator. 

\begin{definition}
In $\mathcal{M}$, define 
\[\fullname {Ord}:=\{(\fullname \alpha, \pi);\ \alpha\in Ord \textrm{ and }\pi\in \Pi\}.\]
We call \emph{full ordinals} the names of the form $\fullname \alpha$ where $\alpha$ is an ordinal in the ground model. 
\end{definition}

\begin{proposition}[\cite{FontanellaM25}]
For every ordinal $\alpha,$ one can realize that $\fullname{\alpha}$ is an $\rlzin$-ordinal (hence an $\in$-ordinal) with a proof-like term that does not depend on $\alpha.$
\end{proposition}

We can expand the language with a predicate for $\fullname{Ord}$ by extending the definition of the truth and falsity values 
in the obvious way for formulas of the form $x\rlzin \fullname {Ord}$, $x\in \fullname {Ord}$ and $x \subseteq \fullname {Ord}$, 
then as proven in \cite{FontanellaM25}, every $\in$-ordinal in the realizability model is extensionally equal to an element of $\fullname {Ord}.$

\begin{proposition}[\cite{FontanellaM25}]
$\Vdash \forall x(x \textrm{ is an $\in$-ordinal }\iff x\in \fullname{Ord})$
\end{proposition}

Our main results will deal with functions on ordinals (choice functions, functions that collapse a cardinal to a smaller ordinal and other examples). In $\mathrm{ZF}_{\rlzin},$ a function may be compatible with the strict identity but not necessarily with the extensional equality: $x\simeq y$ does not imply $f(x)\simeq f(y).$ For this reason it would be useful to work with a representation of ordinals that behaves well with the two equalities, unfortunately full ordinals are not suitable for that: given $x,y\rlzin \fullname \alpha,$ the equality $x\simeq y$ does not imply $x=y.$ To get this feature, another class of representatives for ordinals was introduced in \cite{FontanellaG20}, called 
\emph{hat ordinals}. 

\begin{definition}\label{def: hat ordinals}
%Assume that there are pairwise distinct terms 
We fix a sequence of terms 
$(\underline{\xi})_{\xi<|\mathcal{K}|},$ 
For every ordinal $\alpha\leq |\mathcal{K}|$ in the ground model $\mathcal{M},$ we inductively define 
    \[\hat{\alpha}\dfn\{\langle\hat{\xi},\underline{\xi}\centerdot\pi\rangle\mid\xi<\alpha, \pi\in\Pi\}\]
Names of the form $\hat{\alpha}$ are called \emph{hat ordinals}.
\end{definition}

The ordinal terms $(\underline{\xi})_{\xi<|\mathcal{K}|}$ shall be thought as generalizations of numerals; they may be special instructions (but they are not necessarily proof-like terms). In general, they will be taken pairwise distinct unless we are dealing with the realizability algebra induced by a boolean algebra (see Section \ref{sec: boolean closure}).
%The underlined terms of the form $\underline{\xi}$ act as a generalized version of numerals.
In Definition \ref{def:numerals}, we will specify which conditions must be satisfied by the terms $\underline{\xi}$ for our construction to work.\footnote{In \cite{FontanellaG20}, $\underline{\xi}$ corresponds to the $\xi$-th term in some fixed enumeration of the $\lambda_\mathsf{c}$-terms and in order to realize fragments of the Axiom of Choice in the corresponding model, it is essential that every term can be written as $\underline{\xi}$ for some $\xi<|\mathcal{K}|$. This requirement will not be necessary in this paper.}

As for full names and gimel names, we can define restricted quantifiers for hat ordinals: given a formula $F(x)\in\mathscr{F}_\varepsilon$ and $\alpha<|\mathcal{K}|$, the falsity value of $\forall x^{\hat{\alpha}}\,F(x)$ is defined as
    \[\falsity{\forall x^{\hat{\alpha}}\,F(x)}\dfn\bigcup_{\beta<\alpha}\{\underline{\beta}\centerdot\pi\mid\pi\in\falsity{F[\hat{\beta}]}\}.\]
As before, the formula $\forall x\rlzin \hat{\alpha}\, F(x)\leftrightarrow\forall x^{\hat{\alpha}}\,F(x)$ is realized by a proof-like term which does not depend on $\hat{\alpha}$ nor $F$ (see for instance \cite{FontanellaM25}).

\begin{lemma}\cite[Proposition 3.14]{FontanellaG20}\label{lem: inclusione hat ordinals}
There is a proof-like term that realizes $\hat{\gamma}\subseteq \hat{\beta},$ for every $\gamma\leq \beta.$  
\end{lemma}

%\begin{proof}

    %The formula $\hat{\gamma}\subseteq\hat{\beta}$ is equivalent to 
    %\[
    %    \forall x(x\not\in\hat{\beta}\rightarrow x\not\rlzin\hat{\gamma})\equiv\forall x(\forall y(y\rlzin\hat{\beta},x\simeq y\rightarrow\bot)\rightarrow x\not\rlzin\hat{\gamma})
    %\]
    %therefore considering this latter, a stack in its falsity value, has the form $t\centerdot\underline{\xi}\centerdot\rho$ where $\xi<\gamma$, $t\in\falsitydual{\forall y(y\rlzin\hat{\beta},\hat{\xi}\simeq y\rightarrow\bot)}$, $\rho\in\Pi$. Using Turing fixpoint, one can define a proof-like term $W$
    %(that does not depend on $\gamma,$ nor $\beta$) such that $W\Vdash \forall x(x\simeq x)$ (see for instance \cite{DBLP:journals/Richard}). Moreover, we have $\lambda x. x\underline{\xi}\Vdash \hat{\xi}\rlzin\hat{\beta}$ (see \cite{FontanellaG20}). Thus, 
    %$(t(\lambda x. x\underline{\xi}))W\in\falsitydual{\bot}$, hence $\lambda fn. (f(\lambda x. xn))W$ is a suitable proof-like term realizing the required formula.
%\end{proof}

\begin{lemma} (\cite[Proposition 3.11]{FontanellaG20}, see also \cite{FontanellaM25})\label{lem: hat ordinals are epsilon ordinals}
For every ordinal $\alpha\leq |\mathcal{K}|,$ there is a proof-like term not depending on $\alpha$ that realizes $\hat{\alpha}$ is an $\rlzin$-ordinal.  
\end{lemma}

By adding to the calculus (and to the proof-like terms) a special instruction that we call here \emph{trichotomy operator}, one can realize that hat ordinals are $\rlzin$-TOD sets. 

%We need furthermore to suppose the existence of a trichotomy operator $\chi$ in $\Lambda_\mathsf{c}$, which is used to properly manage accumulators. 

\begin{definition}\label{def:trichotomyoperator}
A trichotomy operator for a sequence of terms $(\underline{\xi})_{\xi<|\mathcal{K}|}$ is a term $\chi\in\Lambda_\mathsf{c}$ such that for any $\alpha,\beta<\kappa$ and any $t,u\in\Lambda_\mathsf{c}, \pi\in\Pi$ we have 
         \[   
           (\chi)\quad \chi\underline{\alpha}\,fg\underline{\beta}\star\pi\succ_{\chi}
                \begin{cases}
                    (f)\underline{\beta}\star\pi&\textrm{if } \beta<\alpha;\\
                    g\star\pi&\textrm{ if $\beta\geq \alpha$}
                \end{cases}
        \]
\end{definition}

Given a term $f,$ the trichotomy operator is used to represent sequences $(f\underline 0, ..., f\underline{\alpha}, \underline{0}, ...\underline{0}...),$ since for instance $\chi \underline{\alpha+1} f \underline{0}\underline{\beta}$ reduces to $f \underline{\beta}$ if $\beta\leq \alpha$ and to $\underline{0}$ otherwise.
%represents the sequence $(f\underline 0, ..., f\underline{\alpha}, \underline{0}, ...\underline{0}...)$.

\begin{lemma}\label{lem: trichotomy operator} 
Let $\mathcal{K}$ be a realizability algebra containing the trichotomy operator $\chi$ as a proof-like term. Then, for every ordinal $\alpha\leq |\mathcal{K}|,$ one can realize that $\hat{\alpha}$ is an $\rlzin$-TOD by a proof-like term that does not depend on $\alpha.$ 
\end{lemma}

This was proven in \cite{FontanellaG20}, however, since our definition of the trichotomy operator is slightly different than the definition of the analogous operator given in \cite{FontanellaG20} (essentially just by a permutation of the terms), we include a proof of this lemma.  

\begin{proof} 
%First we show that $\widehat{\alpha}$ is transitive, for that we show that $\theta=\lambda t.\lambda u.u\Vdash \forall x^{\widehat{\alpha}}\forall y(y\notrlzin \widehat{\alpha}\imp y\notrlzin x).$ Let $\beta<\alpha,$ $c\in \mathcal{M}^{\mathcal{K}},$ $u\in \falsitydual{c\notrlzin \widehat{\alpha}}$ and $\pi\in \falsity{c\notrlzin \widehat{\beta}},$ we want to show that $\theta\ast \underline{\beta}\centerdot u\centerdot \pi\in \Perp.$ We have $\theta\ast \underline{\beta}\centerdot u\centerdot \pi\succ u\ast \pi.$ Moreover, $\pi\in \falsity{c\notrlzin \widehat{\beta}},$ so $\falsity{c\notrlzin \widehat{\beta}}$ is not empty. Thus there exists $\gamma<\beta<\alpha$ such that $c=\widehat{\gamma},$ therefore $\falsity{c\notrlzin \widehat{\beta}}=\{\underline{\gamma}\centerdot\pi';\ \pi'\in \Pi\}=\falsity{c\notrlzin \widehat{\alpha}},$ hence $u\ast \pi\in \Perp.$

%An analogous argument shows that $\rlzin$ defines a strict order on $\widehat{\alpha}$. 

By Lemma \ref{lem: hat ordinals are epsilon ordinals}, we can realize that $\alpha$ is an $\rlzin$-ordinal, we want to realize that $\hat{\alpha}$ is $\rlzin$-totally-ordered. For that, we realize the equivalent formula $\forall x^{\hat{\alpha}}\forall y^{\hat{\alpha}}(x\notrlzin y,  y\not\subseteq x\imp \perp).$ 
%For that, we are going to use the instruction $\chi$. 
Let $\theta$ be the proof-like term in Lemma \ref{lem: inclusione hat ordinals} that realizes $\hat{\gamma}\subseteq \hat{\beta},$ for every $\gamma\leq \beta.$ We let  
$$t:=\lambda b.\lambda c.\lambda u.\lambda v.(\chi b t(v \theta)) c$$ 
and we show that $t\Vdash \forall x^{\hat{\alpha}}\forall y^{\hat{\alpha}}(x\notrlzin y,  y\not\subseteq x\imp \perp).$ Let $\beta,\gamma<\alpha,$ $u\in \falsitydual{\hat{\beta}\notrlzin\hat{\gamma}},$ $v\in \falsitydual{\hat{\gamma}\not\subseteq \hat{\beta}},$ $\pi\in \Pi.$ Either $\beta<\gamma$ or $\gamma\leq \beta.$
In the first case, we have  $t\star\underline{\beta}\centerdot\underline{\gamma}\centerdot u\centerdot v\centerdot \pi\dsucc u\star\underline{\beta}\centerdot \pi\in \Perp.$ 
Otherwise $\gamma\leq\beta,$ hence
$t\star\underline{\beta}\centerdot\underline{\gamma}\centerdot u\centerdot v\centerdot \pi\dsucc v\theta\ast \pi\in \Perp.$ 
\end{proof}

\begin{lemma}[\cite{FontanellaG20}]\label{lem: TODs are good}
In $\mathrm{ZF}_{\rlzin},$ if $\alpha$ is an $\rlzin$-TOD, then for every $x, y \rlzin \alpha,$  
$x\simeq y$ if and only if $x=y$. 
\end{lemma}
It follows that equality and identity coincide on hat ordinals. 

On the other hand there can be only as many hat ordinals as the size of the algebra since they are defined by associating a term to each one of their lower ordinal elements. 

\subsection{Functions in realizability models}

In this paper we will realize principles that involve functions (choice functions, cardinality thus bijections etc.). It is therefore important to discuss the notion of function in the context of the theory $\mathrm{ZF}_{\rlzin}.$ In fact, a function must be compatible with equality, namely $x$ equal to $y$ implies $f(x)$ equal to $f(y)$, but in $\mathrm{ZF}_{\rlzin}$ we deal with two equalities relations. Thus, we shall distinguish the cases of functions either in the sense of extensional
equality or non-extensional equality. Let $(x,y)$ denote the ordered pair of elements $x, y$.

\begin{definition} In the theory $\mathrm{ZF}_{\rlzin},$ let $f$ and $a$ be two sets, we say that   
\begin{itemize}
\item $f$ is an $\in$-function on $a$ if\\ 
$\forall x\rlzin a\exists y\ (x, y)\rlzin f$ and $\forall x, x'\rlzin a \forall y, y'( (x, y)\rlzin f\land (x', y')\rlzin f \land x\simeq x')\imp y\simeq y'$ 
\item $f$ is an $\rlzin$-function on $a$ if\\ $\forall x\rlzin a\exists y\ (x, y)\rlzin f$ and $\forall x\rlzin a \forall y, y'( (x, y)\rlzin f\land (x, y')\rlzin f)\imp y= y'$ 
\end{itemize}
\end{definition}

One can define a (class-)function $\op(\cdot, \cdot):M^\mathcal{K}\times M^\mathcal{K}\to M^\mathcal{K}$ in the ground model, that is interpreted as the ordered pair in realizability models. We omit the details of the definition that can be found in \cite{DBLP:journals/Richard}. We just point out that this (functional) name can be realized to satisfy the expected properties for the ordered pair; in particular, it is compatible with extensional equality, that is the following is realized
%the following formula is realized: 
$$\forall x\forall x'\forall y\forall y'(\op(x,y)\simeq\op(x',y')\leftrightarrow (x\simeq x'\land y\simeq y')).$$ 

The function $\op$ can then be used to `lift' functions in the ground models to $\rlzin$-functions in realizability models.

\begin{definition}
    Let $a\in M^\mathcal{K}$ and $f:\dom{a}\to M^\mathcal{K}$. We define the \emph{lift of $f$} as
        \[\tilde{f}\dfn\{\langle\op(x,f(x)),\pi\rangle\mid\langle x,\pi\rangle\in a\}.\]
\end{definition}

\begin{proposition}[\cite{DBLP:journals/Richard}]\label{prop:functionlift} 
    There exists a term $\theta_1\in\qproofs$ such that, for all $a\in M^\mathcal{K}$ and
    $f:\dom{a}\to M^\mathcal{K}$, we have $\theta_1\Vdash \tilde{f}\;\text{is an $\rlzin$-function}$.    
%    (i.e. $\forall x, y, y'(\op(x, y)\rlzin f, \op(x, y')\rlzin f\imp y=y')$) 
\end{proposition}

%The proof is straightforward, we omit it. The term $\theta_1$ corresponds to $\lambda xyz.zy$ and it will be employed in the proof of \cref{thm:preserving cardinals}.

\begin{comment}
    \begin{proof}
    We show that the formula $\forall x\rlzin a \forall y,y'(y\neq y'\rightarrow\op(x,y)\rlzin\tilde{f}\rightarrow\op(x,y')\notrlzin\tilde{f})$ is realized. A stack in the falsity value of the formula has the form $t\centerdot u\centerdot v\centerdot\pi$ where, for some names $c, d, d'$, $t\in\falsitydual{c\rlzin a}, u\in\falsitydual{d\neq d'}, v\in\falsitydual{\op(c,d\rlzin \tilde{f}}, \pi\in\falsity{\op(c,d')\notrlzin\tilde{f}}$. By $\falsity{\op(c,d')\notrlzin\tilde{f}}\neq\emptyset$ and the definition of $\tilde{f}$, we can infer that $c\in\dom{a}$ and $d'=f(c)$. We can then distinguish two cases: if $d=f(c)$, then $u\in\falsitydual{\bot}$; otherwise, $v\in\falsitydual{\top\rightarrow\bot}$. In both cases $(v)u\in\falsitydual{\bot}$, thus $\lambda xyz.zy$ is a realizer for the formula. 
\end{proof}
\end{comment}

In the following, \emph{par abus de langage}, given a function $f$ from an ordinal $\gamma$ to $M^\mathcal{K}$, we denote $\tilde{f}$ the lift of the function $f_0:\dom{\hat{\gamma}}\to M^\mathcal{K}$ defined as $f_0(\hat{\xi})=f(\xi)$ for $\xi<\gamma$. In particular, in \cref{sec: ACkappa}, the lift will be used to represent functions from $\hat{\gamma}$ to some greater ordinal $\hat{\mu}$. 

\begin{proposition}\label{prop:domain of a function}
    If $f:\gamma\to M^{\mathcal{K}}$, $\gamma<\kappa$, then $\Vdash\forall x^{\hat{\gamma}}\exists y(\op(x,y)\rlzin\tilde{f})$ and $\Vdash \forall x(\exists y(\op(x,y)\rlzin \tilde{f})\rightarrow x\rlzin\hat{\gamma}$. Therefore, the lift of $f$ has domain $\hat{\gamma}$.
\end{proposition}
\begin{proof}
    The formula $\forall x^{\hat{\gamma}}\exists y(\op(x,y)\rlzin\tilde{f})$ can be rewritten as $\forall x^{\hat{\gamma}}(\forall y(\op(x,y)\notrlzin\tilde{f})\rightarrow\bot)$. Consider a stack $\pi\in\falsity{\forall x^{\hat{\gamma}}(\forall y(\op(x,y)\notrlzin\tilde{f})\rightarrow\bot)}$, which has the form $\underline{\alpha}\centerdot t\centerdot\pi'$ for some $\alpha<\gamma$, $t\in\falsitydual{\forall y(\op(\hat{\alpha},y)\notrlzin\tilde{f})}$, and $\pi'\in\Pi$. Notice that $t$ is contained by hypothesis in $\falsitydual{\op(\hat{\alpha},f(\alpha))\notrlzin\tilde{f}}$, and by definition of $\tilde{f}$, $\falsity{\op(\hat{\alpha},f(\alpha))\notrlzin\tilde{f}}=\{\underline{\alpha}\centerdot\varsigma\mid\varsigma\in\Pi\}$. Therefore $t\star\underline{\alpha}\centerdot \pi'\in\Perp$ (as $\underline{\alpha}\centerdot \pi'\in\falsity{\op(\hat{\alpha},f(\alpha))\notrlzin\tilde{f}}$) and so is the process $\lambda xy.yx\star\underline{\alpha}\centerdot t\centerdot\pi'$. Hence, $\lambda xy.yx$ realizes the first formula.

    For the second part of the statement, the formula $\forall x(\exists y(\op(x,y)\rlzin \tilde{f})\rightarrow x\rlzin\hat{\gamma})$ equivalent to $\forall x(x\notrlzin\hat{\gamma}\rightarrow \forall y(\op(x,y)\notrlzin \tilde{f})$. Consider a stack $\rho$ in the falsity value of the formula. Then, $\rho$ has the form $u\centerdot\rho'$ for $u\in \falsitydual{c\notrlzin\hat{\gamma}}$, $\rho'\in\falsity{\forall y(\op(c,y)\notrlzin \tilde{f})}$, for some $c\in M^{\mathcal{K}}$. There are two cases, depending on $c$.
    \begin{itemize}
        \item If $c=\hat{\alpha}$ for some $\alpha<\gamma$, then $\falsity{c\notrlzin\hat{\gamma}}=\{\underline{\alpha}\centerdot\varsigma\mid\varsigma\in\Pi\}$ and there exists (a unique) $d\in M^\mathcal{K}$ such that $f(\alpha)=d$. This latter observation implies that $\falsity{\forall y(\op(c,y)\notrlzin \tilde{f})}=\falsity{\op(c,d)\notrlzin \tilde{f})}=\falsity{c\notrlzin\hat{\gamma}}$, the last equation holding by definition of $\tilde{f}$, thereby $u\star\rho'\in\Perp$.
        \item If $c$ is not of the form discussed in the previous point, then for all $d\in M^\mathcal{K}$, $\falsity{\op(c,d)\notrlzin\tilde{f}}$ is empty, and so is $\falsity{\forall y(\op(c,y)\notrlzin \tilde{f})}$. Thus, $\falsity{c\notrlzin\hat{\gamma}\rightarrow\forall y(\op(c,y)\notrlzin \tilde{f})}=\emptyset$.
    \end{itemize}
    In the first case $I\star u\centerdot\rho'\in\Perp$, and also in the second the identity is trivially a realizer. Hence, the identity realizes the formula. 
\end{proof}
%%%

\section{Fully closed realizability algebras}\label{sec:closure}

When constructing a $\mathrm{ZF}$-model with forcing, it may happen that new bijections from a cardinal to a lower ordinal are added in the resulting forcing extension, thus a cardinal may \emph{collapse} to a smaller ordinal and no longer remain a cardinal in the forcing extension. The same may occur when we build classical realizability models for set theory: there is no guarantee that the representative $\hat{\kappa}$ remains a cardinal even if $\kappa$ is a cardinal in the ground model. This phenomenon affects our construction if we want to realize uncountable versions of AC, since relevant cardinals may collapse to countable ordinals. In the context of forcing, some requirements may be used to ensure that cardinals are preserved. One of those properties is $\kappa$-closure, that guarantees the preservation of cardinals below $\kappa$. In this section, we present an analogue of the $\kappa$-closure for classical realizabilty algebras.

In forcing, a poset $(\mathbb{P}, \leq)$ has the $\kappa$-closure property when for every decreasing sequence $(p_\alpha)_{\alpha<\gamma}$ of elements of $\mathbb{P}$ with $\gamma<\kappa,$ there is $p\in\mathbb{P}$ such that $p\leq p_\alpha$ for any $\alpha<\gamma$. 

In order to implement this in our formalism, we need to introduce some definitions to encode (possibly transfinite) sequences of terms. 

%assume the existence of representatives for ordinals below $\kappa$ in $\Lambda_\mathsf{c}$. 

%In \cref{subsec: names}, we discussed full ordinals and hat ordinals as a way of representing ordinals in realizability models. However, when we consider an infinite cardinal $\kappa$ in the ground model, it may happen that $\fullname{\kappa}$ (resp. $\hat{\kappa}$) is no longer a cardinal in the realizabilty model, or that while $\kappa$ was uncountable in the ground model, $\fullname{\kappa}$ (resp. $\hat{\kappa}$) become countable in the realizability model. In this case, we say that the cardinal is \emph{collapsed}. In forcing, several properties have been studied that prevent cardinals from collapsing, the most relevant being the $\kappa$-chain condition and the $\kappa$-closure. The former ensures preservation of cardinals greater than or equal to $\kappa$, while the latter ensures preservation of cardinals up to $\kappa$. In \cite{FontanellaG20} and \cite{FontanellaM25} some properties analogous to the $\kappa$-chain condition are developed for realizability models. In this section, we discuss an analogue of the $\kappa$-closure for realizability algebras. 

We define the notion of accumulator for any sequence of terms.

\begin{definition}\label{def:accumulator}
    Given a sequence $(u_\xi)_{\xi<\gamma}$ for some ordinal $\gamma$, a term $t\in\Lambda_\mathsf{c}$ \emph{accumulates the sequence}  $(u_\xi)_{\xi<\gamma}$ if for any $\xi<\gamma$, $(t)\underline{\xi}\dsucc t_\xi$. 
\end{definition}

Accumulators are used to represent sequences of terms within $\Lambda_\mathsf{c}.$ For $\kappa=\aleph_0$, they are the programs stocking a finite sequence and outputting the $i^{th}$ element on input $\underline{i}$.

The first part of $\kappa$-\closure will consists in requiring the existence of accumulators for sequences of length less that $\kappa$. In addition to that, the property will impose that these accumulators will not be in the truth value of $\bot$ unless the sequence they represent already contained `incompatible' terms.

%\todo{Useless definition?}
%\begin{definition}
%    Given $u\in\Lambda_\mathsf{c}$ and $v\in\Lambda_\mathsf{c},$ we say that $v$ is \emph{compatible} with $u$ if $v=u$ or if $uv\notin \falsitydual{\bot}.$   
 %   We say that $v$ is \emph{incompatible} with $u$ if it is not compatible.
%\end{definition}

%\begin{remark}  
%    Note that compatibility is not symmetric: we may have $uv\notin \falsitydual{\bot},$ but $vu\in\falsitydual{\bot}.$ Moreover, since in \cref{sec: classical realizability} we imposed as an additional condition on the pole that if $(t)u\star \pi\in \Perp$ then $t\star u\centerdot \pi\in \Perp,$ we shall note that $v$ is \emph{compatible} with $u$ if and only if $u=v$ or there exists $\pi\in\Pi$ such that $u\star v\centerdot\pi\not\in\Perp.$ 
%\end{remark}

Using the definitions above, we define the notion of $\kappa$-\closure for realizability algebras.

\begin{definition}\label{def: kappa-closure} 
Let $\mathcal{K}=\langle\Lambda_\mathsf{c},\Pi,\succ,\Perp\rangle$ be an uncountable realizability algebra and $\kappa$ a regular cardinal in the ground model $\mathcal{M}$ such that $\omega<\kappa\leq |\mathcal{K}|,$ we say that $\mathcal{K}$ is \emph{$\kappa$-\closed} if for every limit ordinal $\gamma<\kappa$, and for every sequence $(u_\xi)_{\xi<\gamma}\subseteq \Lambda_\mathsf{c}^{\textrm{closed}}$ there is a term $u\in\Lambda_\mathsf{c}^{\textrm{closed}}$ such  that
    the following properties are satisfied: 
    \begin{enumerate}
        \item $u$ accumulates $(u_\xi)_{\xi<\gamma}$;
        \item for every term $v\in\Lambda_\mathsf{c}$ such that $\FV{v}\subseteq\{x\}$, 
        if $v\{x:=u\}\in\falsitydual{\bot}$, then there exists $\delta<\gamma$ such that for every accumulator $w$ of $(u_\xi)_{\xi<\delta},$ we must have $v\{x:=w\}\in\falsitydual{\bot}$.
        %\item given a term $v$, if for any $\delta<\gamma$ there exists an accumulator $w$ of the sequence $(u_\xi)_{\xi<\delta}$ such that $v\{x:=w\}\not\Vdash\bot$, then for any accumulator $t$ of the sequence $(u_\xi)_{\xi<\gamma}$, $v\{x:=t\}\not\Vdash\bot$.
    \end{enumerate}

We will call $u$ a \emph{minimal accumulator of} $(u_\xi)_{\xi<\gamma}.$ 
    
\end{definition}

\begin{remark}\label{rmk:kclosure}
    We will often use the contrapositive of the second item which corresponds to the following: given a term $v$, if for any $\delta<\gamma$ there exists an accumulator $w$ of the sequence $(u_\xi)_{\xi<\delta}$ such that $v\{x:=w\}\notin \falsitydual{\bot}$, then we must have $v\{x:=u\}\notin\falsitydual{\bot}$.
\end{remark}

%Note that item 3. in Definition \ref{def:numerals} implies that for every limit ordinal $\delta,$ we have $\underline{\delta}$ is a minimal accumulator of the sequence $(\underline{\xi+1})_{\xi<\delta}.$

\subsection{Closure in boolean algebras}\label{sec: boolean closure}

This part is devoted to showing the connection between the $\kappa$-\closure as stated in \cref{sec:closure} and the $\kappa$-closure property in  forcing. 

Given a complete boolean algebra $\mathbb{B},$ we denote by $\mathbb{B}^+$ the set $\mathbb{B}\setminus\{0\}$ and $\leq$ the usual order relation induced by $\land$, i. e. for any $a, b\in B$, $a\leq b$ if and only if $a\land b= a$.

\begin{definition}
We recall that a \emph{dense} subset $D$ of a complete boolean algebra $\mathbb{B}$ is a subset $D\subseteq \mathbb{B}^+$ such that for every $b\in \mathbb{B},$ there is $d\in D$ such that $d\leq b$.  
\end{definition}

\begin{definition} \cite[p.68]{Bell2005} Given a complete boolean algebra $\mathbb{B}$ and a cardinal $\kappa,$ a dense subset $D\subseteq \mathbb{B}^+$ is \emph{$\kappa$-closed} if for every $\gamma<\kappa$ and every $\leq$-decreasing sequence $(d_\xi)_{\xi<\gamma}$ of $D,$ there is $d\in D$ such that for every $\xi<\gamma,$ $d\leq d_\xi.$
\end{definition}

The $\kappa$-closure ensures preservation of cardinals up to $\kappa.$

\begin{proposition} \cite[p.68]{Bell2005} If $\mathbb{B}$ has a dense subset which is $\kappa$-closed, then for every cardinal $\mu\leq \kappa,$ $\mu$ remains a cardinal in the corresponding boolean valued model $V^{\mathbb{B}}.$
\end{proposition}

\begin{lemma}\label{lem: equivalence closure in BA}
The following are equivalent: 
\begin{enumerate}
\item $\mathbb{B}^+$ is $\kappa$-closed  
\item for every limit ordinal $\gamma<\kappa$, if for every $\delta<\gamma,$ we have $\bigwedge_{\xi<\delta} b_\xi\neq\mathbf{0}$, then $\bigwedge_{\xi<\gamma} b_\xi\neq\mathbf{0}$. 
\end{enumerate}
\end{lemma}

\begin{proof}
(1) implies $(2):$ for every $\delta<\gamma,$ let $d_\delta:= \bigwedge_{\xi<\delta} b_\xi$. The sequence $(d_\delta)_{\delta<\gamma}$ is  decreasing and by hypothesis is contained in $\mathbb{B}^+$, therefore it has a lower bound $d$ in $\mathbb{B}^+$ by $\kappa$-closure. We have
$d\leq\bigwedge_{\delta<\gamma}d_\delta = \bigwedge_{\xi<\gamma} b_\xi$ and $d\neq {\bf0}$ hence $\bigwedge_{\xi<\gamma} b_\xi\neq {\bf 0}.$

(2) implies $(1):$ let $(d_\delta)_{\delta<\gamma}$ be a decreasing sequence in $\mathbb{B}^+.$ Then for every $\delta<\gamma,$ we have  
$\bigwedge_{\xi< \delta} d_\xi\geq d_\delta\in \mathbb{B}^+$. If $\gamma$ is a successor ordinal $\alpha+1,$ then $d_\alpha$ is a lower bound of the sequence contained in $\mathbb{B}^+$. If $\gamma$ is a limit ordinal, we can apply $(2)$, hence $d\dfn \bigwedge_{\xi<\gamma} d_\xi$ is a lower bound in $\mathbb{B}^+$.
\end{proof}

Any complete boolean algebra $\mathbb{B} = \langle B, \mathbf{0}, \mathbf{1}, \land, \vee, \neg\rangle$, naturally induces a realizability algebra $\mathcal{K}(\mathbb{B})=\langle\Lambda_\mathsf{c} , \Pi, \succ_{\mathbb{B}}, \Perp_{\mathbb{B}}\rangle$ as follows (see also \cite{FontanellaG20}). We let $\Pi_0 \dfn B$ so that 
the set of atomic stacks is $B,$ and define a \emph{boolean-value} function $(\cdot)^{\mathbb{B}}:\Lambda_\mathsf{c}\cup\Pi\to\mathbb{B}$ by induction on the structure of the elements of $\Lambda_\mathsf{c}\cup\Pi$. For any variable $x$, any terms $t, u\in\Lambda_\mathsf{c}$, any atomic stack $\pi_0\in\Pi_0$ and any stack $\pi\in\Pi$, we set 
\[\pi_0^{\mathbb{B}} \dfn \pi_0,\quad (t\centerdot\pi)^{\mathbb{B}} \dfn t^{\mathbb{B}}\land\pi^{\mathbb{B}};\]
    \[x^{\mathbb{B}} \dfn \cc^{\mathbb{B}} \dfn \mathbf{1},\quad k_\pi^{\mathbb{B}} \dfn \pi^{\mathbb{B}},\quad
    (\lambda x.t)\dfn t^{\mathbb{B}},\quad ((t)u)^{\mathbb{B}} \dfn t^{\mathbb{B}}\land u^{\mathbb{B}};
\]
    
%\todo{We interpret any additional instruction of interest $w\in\Lambda_\mathsf{c}$ (in particular the terms $\underline{\alpha}$ and $\chi$ as defined in \cref{sec:closure}) as the top element of $\mathbb{B}$, i. e. $w^\mathbb{B}=\mathbf{1}$. }

The function $(\cdot)^{\mathbb{B}}$ can be further extended to the set of processes $\Lambda_\mathsf{c}\star\Pi$ by defining $(t\star\pi)^{\mathbb{B}} \dfn t^{\mathbb{B}}\land\pi{^\mathbb{B}}$.
Then, the relation $\succ_{\mathbb{B}}$ is defined by 
    \[t\star\pi\succ_{\mathbb{B}} u\star\rho\;\,\Leftrightarrow\,\;(t\star\pi)^{\mathbb{B}}\leq (u\star\rho)^{\mathbb{B}}\]
for $t\star\pi, u\star\rho\in\Lambda_\mathsf{c}\star\Pi.$ It is easy to check that this evaluation satisfies the rules \emph{push}, \emph{grab}, \emph{save} and \emph{restore}. To conclude the construction, we define $\Perp_{\mathbb{B}}=\{t\star\pi\in\Lambda_\mathsf{c}\star\Pi\mid (t\star\pi)^{\mathbb{B}} = \mathbf{0}\}$. In particular, this implies that $t\in\falsitydual{\bot}$ if and only if $t^\mathbb{B}=\mathbf{0}$.

Since we employ explicit substitutions to state \closure for realizability algebras, thus we need to define the boolean value associated to a term in the form $v\{x:=t\}$. The definition of $v\{x:=t\}^{\mathbb{B}}$ comes from the observation that for any variable $x$, its associated value is $\mathbf{1}$, together with the fact that an easy induction proves the equality $t^{\mathbb{B}}=\bigwedge\{u^{\mathbb{B}}\mid u\in\subterm{t}\}$. Given $t, v$ terms, the value of $v\{x:=t\}^{\mathbb{B}}$ is defined as
    \[v\{x:=t\}^{\mathbb{B}}\dfn
    \begin{cases}
        v^{\mathbb{B}}\land t^{\mathbb{B}}& x\in\subterm{v};\\
        v^{\mathbb{B}}& x\not\in\subterm{v}.
    \end{cases}
    \]

We show that $\mathbb{B}^+$ is $\kappa$-closed for a regular uncountable cardinal $\kappa$ if and only if $\mathcal{K}(\mathbb{B})$ is $\kappa$-\closed, but we need to specify who are the ordinal terms in $\mathcal{K}(\mathbb{B}).$ Any sequence of ordinals terms that are interpreted as $\mathbf{1}$ would work, for instance we may
define every $\underline{\alpha}$ as the identity $I;$ then every $\alpha^\mathbb{B}$ is $\mathbf{1}.$

%\begin{enumerate}
%\item when $\alpha$ is a limit ordinal (including $0$) $\underline{\alpha}$ is $I$ 
%\item $\underline{\alpha+1}\dfn(\underline{s}) \underline{\alpha}$ 
%where $\underline{s}=I$
%\end{enumerate}

\begin{theorem}\label{thm: equivalence of closure} Suppose that $\kappa$ is a regular uncountable cardinal, then  
    $\mathbb{B}^+$ is $\kappa$-closed if and only if $\mathcal{K}(\mathbb{B})$ is $\kappa$-\closed.
\end{theorem}
\begin{proof}
    Suppose that $\mathbb{B}^+$ is $\kappa$-closed. Fix a limit ordinal $\gamma<\kappa$ and consider a sequence $(u_\xi)_{\xi<\gamma}$ in $\mathcal{K}(\mathbb{B})$. Set $\pi=\bigwedge_{\xi<\gamma} u_{\xi}^{\mathbb{B}}$ and define $t\dfn k_{\pi}$. By definition, $t^{\mathbb{B}}=\bigwedge_{\xi<\gamma} u_{\xi}^{\mathbb{B}}\leq u_\xi^{\mathbb{B}}$ for every $\xi<\gamma.$ It follows that $t$ accumulates the sequence: 
    indeed, for every stack $\rho$ and every ordinal $\xi<\gamma,$ we have $(t\underline{\xi}\star\rho)^\mathbb{B}= t^\mathbb{B}\land \mathbf{1}\land  \rho^\mathbb{B}=\left(\bigwedge_{\xi<\gamma} u_{\xi}^{\mathbb{B}}\right)\land \rho^\mathbb{B}\leq u_\xi^{\mathbb{B}}\land \rho^{\mathbb{B}}$, hence $t\underline{\xi}\star \rho\succ u_\xi\star \rho.$ For the second part of the requirement, 
    we fix a term $v\in \mathcal{K}(\mathbb{B}).$ 
     Suppose that for every $\delta<\gamma$ there is $w_\delta$ an accumulator of the sequence $(u_\xi)_{\xi<\delta}$ such that $v\{x:=w_\delta\}\notin \falsitydual{\bot}$, that is $v\{x:=w_\delta\}^{\mathbb{B}}\neq\mathbf{0}$. If $x\not\in\subterm{v}$, then for all $u$, $v\{x:=u\}^{\mathbb{B}}=v\{x:=w_\delta\}^{\mathbb{B}}=v^{\mathbb{B}}\neq\mathbf{0}$, and the result is straightforward. 
     Suppose then that $x\in\subterm{v}$. For all $\delta<\gamma$, $v\{x:=w_\delta\}^{\mathbb{B}}=v^{\mathbb{B}}\land w_\delta^\mathbb{B}\neq \mathbf{0}.$ Since $w_\delta$ is an accumulator, we have $w_\delta^\mathbb{B}\leq (\bigwedge_{\xi<\delta}u_\xi)^\mathbb{B}$, it follows that $v\land (\bigwedge_{\xi<\delta}u_\xi)^\mathbb{B} \neq \mathbf{0}$. For every $\delta<\gamma,$ let $d_\delta:= v^\mathbb{B}\land (\bigwedge_{\xi<\delta}u_\xi)^\mathbb{B}.$ Then $(d_\delta)_{\delta<\gamma}$ is a decreasing sequence of elements of $\mathbb{B}^+.$ Since $\mathbb{B}^+$ is $\kappa$-closed, we have $\bigwedge_{\delta<\gamma} d_\delta\neq 0.$
    %the sequence $(d_\delta)_{\delta<\gamma}$ has a lower bound $d\in \mathbb{B}^+.$ 
    We have $v^{\mathbb{B}}\land t^{\mathbb{B}}= v^\mathbb{B}\land (\bigwedge_{\xi<\gamma} u_\xi)^\mathbb{B} = \bigwedge_{\delta<\gamma} d_\delta\neq \mathbf{0},$ hence $v\{x:=t\}\notin\falsitydual{\bot}.$

    Suppose now that $\mathcal{K}(\mathbb{B})$ is $\kappa$-\closed, we show that $\mathbb{B}^+$ is $\kappa$-closed by contrapositive. Consider a sequence $(b_\xi)_{\xi<\gamma}\subseteq B$ for some limit $\gamma<\kappa.$ We show that if $\bigwedge_{\xi<\gamma} b_\xi=\mathbf{0},$ then there is $\delta<\gamma$ such that $\bigwedge_{\xi<\delta} b_\xi=\mathbf{0}.$
    For every $\delta\leq \gamma,$ set $d_\delta\dfn \bigwedge_{\xi<\delta} b_\xi$ and let $k_{\delta}\dfn k_{d_\delta}\in\Lambda_\mathsf{c}.$

    Note that every $k_\delta$ accumulates $(k_\xi)_{\xi<\delta}:$
    indeed for $\xi<\delta\leq \gamma,$ we have 
    $(k_\delta\underline{\xi}\star \centerdot\rho)^\mathbb{B} = d_\delta\land \rho^\mathbb{B} \leq d_\xi \land \rho^\mathbb{B}$ (since $d_\delta\leq d_\xi$) hence $ k_\delta\underline{\xi}\star\rho\succ_{\mathbb{B}} k_\xi\star\rho.$ 
Let $u$ be a minimal accumulator of $(k_\xi)_{\xi<\delta}$ given by the $\kappa$-\closure, then $u^{\mathbb{B}}\leq d_{\gamma}$ : indeed $u^{\mathbb{B}} = (u \underline{\delta})^{\mathbb{B}}$ for every $\delta<\gamma,$ and $u\underline{\delta}\dsucc k_\delta$ hence $u^{\mathbb{B}}\leq \bigwedge_{\xi<\gamma} b_\xi=d_\gamma.$ It follows that $u\succ_{\mathbb{B}} k_\gamma.$

    %In particular, $k_\gamma$ accumulates the sequence $(k_\xi)_{\xi<\gamma}.$
    
    By hypothesis $\bigwedge_{\xi<\gamma} b_\xi$ (that is $d_\gamma$) is equal to $\mathbf{0}$, then $k_\gamma\in\falsitydual{\bot},$ 
    hence $u\in\falsitydual{\bot}$.
    It follows that for $v\dfn x$, we have $v\{x:=u\}\in\falsitydual{\bot}$. 
    %On the other hand, $k_\gamma$ accumulates the sequence $(k_\xi)_{\xi<\gamma},$ hence 
    By $\kappa$-\closure, there is $\delta<\gamma$ such that 
    $v\{x:=k_\delta\}\in \falsitydual{\bot},$ that is $k_\delta\in \falsitydual{\bot}.$ It follows that  
    $\bigwedge_{\xi<\gamma} b_\xi=d_\delta=k_\delta^{\mathbb{B}}=\mathbf{0},$ which concludes the proof.
\end{proof}

\section{Transfinite bar-induction}\label{sec: bar induction} 

We introduce now a generalization of bar-induction: while the usual bar-induction can be seen as an induction principle for trees of finite sequences of natural numbers, our transfinite bar-induction is an induction principle for trees of sequences of length below $\gamma,$ where $\gamma$ is a limit ordinal (thus such sequences are not necessarily finite).
Then, we introduce a term which generalizes the bar-recursion operator of \cite{DBLP:conf/csl/Krivine16} to our transfinite case.

Consider a set $A$ and the set $A^{\leq\gamma}$ of sequences in $A$ of length at most $\gamma$. For $s\in A^{\leq\gamma}$, we denote by $|s|$ the length of the sequence (or formally its domain), and by $s\restr\alpha$ the sequence $s'\in A^\alpha$ such that $s'(\xi)=s(\xi)$ for all $\xi<\alpha\leq|s|$. Given an element $x\in A$, the extension of $s$ with $x$ will be denoted by $s\smallfrown x$, which is the sequence $s'$ such that $|s'|=\alpha+1$, $s'\restr\alpha=s$, $s'(\alpha)=x$ (here $|s|=\alpha<\gamma$). For $s, t\in A^{<\gamma},$ we write $s\sqsubseteq t$ when $|s|\leq |t|$ and $t\restr |s|= s.$ Note that when $(s_\alpha)_{\alpha<\delta}$ is a $\sqsubseteq$-increasing sequence of elements of $A^{\leq \gamma},$ then $s:=\bigcup_{\alpha<\delta} s_\alpha$ is a sequence of domain $\bigcup_{\alpha<\delta} \dom{s_\alpha}$ such that $s(\xi) = s_\alpha(\xi)$ for every $\alpha$ such that $\xi\in \dom{s_\alpha}.$  

\begin{definition} Given a limit ordinal $\gamma,$ we introduce the following notions. 
\begin{itemize}
\item A predicate $T(x)$ on $A^{<\gamma}$ is a \emph{tree on $A^{<\gamma}$} if $\exists s\in A^{<\gamma}(T(s))$ ($T$ is not empty) and 
$\forall s\in A^{<\gamma}(T(s)\imp \forall \xi<|s| (T(s\restr\xi)))$ 
\item A \emph{$\gamma$-branch for $T(x)$} is a sequence $b\in A^{\gamma}$ such that 
$\forall \alpha<\gamma\ T(b\restr \alpha).$
\item A \emph{bar} $B(x)$ for $T(x)$ is a predicate on $A^{<\gamma}$ such that $B(x)\imp T(x)$ and for every $\gamma$-branch $b$ for $T(x),$ there exists $\delta<\gamma$ such that $B(b\restr\delta)$.
\item A bar $B(x)$ for $T$ is \emph{inductive} if the following conditions hold:
%$\forall s\in A^{<\gamma} (\forall t (T(t)\land T(s)\land s\sqsubseteq t)\imp B(t))\imp B(s).$
\begin{enumerate}
\item $\forall s\in A^{<\gamma}(T(s)\land \forall a\in A(T(s\smallfrown a)\imp B(s\smallfrown a))\rightarrow B(s))$;
\item for all increasing sequences $(s_\alpha)_{\alpha<\delta}$ with $\delta<\gamma,$  $B(\bigcup_{\alpha<\delta} s_\alpha)\imp \exists \alpha<\delta\ B(s_\alpha)$. 
\end{enumerate}

\end{itemize}
\end{definition}

%A \emph{bar} $B(s)$ for $T(s)$ is a predicate on $A^{<\kappa}$ such that for each $s\in A^{\leq\kappa}$ satisfying $T(s)$, there exists $\delta\leq|s|$ such that $B(s\restr\delta)$. 
\begin{definition}\label{def: transfinite bar induction}
The \emph{transfinite bar-induction principle for a limit ordinal $\gamma$} (in symbols $\mathrm{BI}_\gamma$) states that for any set $A$, given a tree $T$ on $A^{<\gamma}$ and a bar $B$ for $T$, if the tree is closed by $\sqsubseteq$-increasing sequences of length $<\gamma$ and the bar is inductive, then the empty sequence satisfies the bar. 
\end{definition}
This principle generalizes bar-induction as formulated in \cite[Sec. 3]{Berger_Oliva_2005} and \cite[p. 144]{Luckhardt_1973}. We show that 
this principle is equivalent to a choice principle known as 
\emph{Higher Dependent Choice} for a limit ordinal $\gamma,$ and denoted $\mathrm{DC}_\gamma$ (see for instance \cite{holy2025}): for every set $A$ and every tree $T$ over $A^{<\gamma},$ if $T$ has no terminal nodes (i.e. $\forall s\ T(s)\imp \exists a\in A\ T(s\smallfrown a)$) and $T$ is closed under $\sqsubseteq$-increasing sequences of length less than $\gamma,$ then it has a $\gamma$-branch $b.$ 
We denote $\mathrm{DC}_{<\kappa}$ the principle $\forall \gamma<\kappa(\gamma\;\textrm{limit} \rightarrow\mathrm{DC}_\gamma).$ 

%NON CANCELLARE QUESTO COMMENTO: 
%L'articolo da cui abbiamo preso questa %definizione di DC è questo %https://arxiv.org/pdf/2510.14821

\begin{lemma}
    For every limit ordinal $\gamma,$ the transfinite bar-induction principle for $\gamma$ is provable in $\mathrm{ZF + DC_{\gamma}}$.
\end{lemma}
\begin{proof}
Let $T(x)$ be a tree on some $A^{<\gamma}$ closed by $\sqsubseteq$-increasing sequences of size $<\gamma$ and let $B(x)$ be an inductive bar for $T$; suppose by contradiction that the empty sequence $\langle\,\rangle$ satisfies $\neg B(\langle\,\rangle)$. Define 
$$T'(s):\iff \forall t\sqsubseteq s (T(t) \land \lnot B(t)),$$ then 
$T'$ defines a tree over $A^{<\gamma}.$ $T'$ has no terminal nodes since $B$ is inductive (contrapositive of item 1). Moreover,  
if $(s_\alpha)_{\alpha<\delta}$ is a $<\gamma$-increasing-sequence of nodes of T', then $\bigcup (s_\alpha)_{\alpha<\delta}$ is in T since T is closed by $<\gamma$-increasing-sequences of nodes, and it is not in the bar since each $s_\alpha$ is not in the bar (contrapositive of item 2).
By $\mathrm{DC}_\gamma,$ $T'$ has a $\gamma$-branch $b,$ contradicting $B$ is a bar for $T.$ \end{proof}

\begin{lemma}
    For every limit ordinal $\gamma,$ $\mathrm{ZF}$ proves that $\mathrm{BI}_\gamma$ implies $\mathrm{DC}_\gamma$.
\end{lemma}
\begin{proof}
    Consider a tree $T$ on $A^{<\gamma}$ with no terminal nodes and closed by $\sqsubseteq$-sequences of length $<\gamma$ and suppose that there exists no $\gamma$-branch for such $T.$ 
    %(then in particular $T$ is non empty otherwise any $\gamma$-sequence of elements of $A$ is a $\gamma$-branch). 
    Then, the predicate $B(x):\equiv x\in A^{<\gamma}\land \lnot T(x)$ is a bar for $A^{<\gamma}$ seen as a tree: indeed for any $\gamma$-branch $b$, since no $\gamma$-branch is in $T$, then there exists $\delta<\gamma$ such that $\neg T(b\restr\delta)$. We show that $B$ is inductive: for the first condition, we observe that since $T$ has no terminal nodes, we have $\forall s\in A^{<\gamma}(\forall a(\lnot T(s\smallfrown a)\imp \lnot T(s));$ for the second condition, we observe that since $T$ is closed by $\sqsubseteq$-increasing sequences of length $\delta<\gamma,$ we have $\lnot T(\bigcup_{\alpha<\delta}s_\alpha)\imp \exists \alpha<\delta \lnot T(s_\alpha).$ $BI_\gamma$ implies that  $\neg T(\langle\,\rangle),$ but $T$ is non empty (by definition of tree) thus $T(\langle\,\rangle),$ a contradiction. 
\end{proof}
We introduce now the computational counterpart of our transfinite bar-induction in the context of classical realizability. Our transfinite bar-recursion operator is a natural generalization of Krivine's bar-recursion operator as defined in \cite{DBLP:conf/csl/Krivine16}. 

\begin{definition}\label{def: transfinite bar recursion}
Given a realizability algebra that contains a trichotomy operator $\chi$. The \emph{transfinite bar-recursion operator} is defined in $\Lambda_\mathsf{c}$ as the term  
        \[\BR:=\lambda gu.(Y)\lambda hif.(u)(\chi i f)(g)\lambda z.(h (\underline{s})i)(\chi)ifz\]
\end{definition}
To ease the readability, we introduce some notation.
\begin{definition}
    We denote by $\Restr[f, \underline{\alpha}]$ the term $\chi \underline{\alpha}f$, and $\Ind[g, u, f, \underline{\alpha}]$ denotes the term $\lambda z.(\BR gu \underline{\alpha+1})(\chi)\underline{\alpha} f z.$  
\end{definition}

This notation is motivated by the fact that $\chi \underline{\alpha} f$ represents some sequence $f$ restricted to an ordinal $\alpha$ as explained in the comment after Definition \ref{def:trichotomyoperator}, while $\lambda z.(\BR gu \underline{\alpha+1})(\chi)\underline{\alpha} f z$ is the recursive call of our transfinite recursion operator.

With this notation, the reduction of $\BR$ on terms $G, U,\underline{\alpha},t\in\Lambda_\mathsf{c}, \alpha<\kappa$ can be written as
    \begin{equation}\label{eq:BRreduction}
        \BR GU\underline{\alpha}t\dsucc (U)(\Restr[t, \underline{\alpha}])(G)\Ind[G, U, t, \underline{\alpha}].
    \end{equation}

Intuitively, the program considers the current approximation $t\restr \alpha$ (i.e. $\Restr[t, \underline{\alpha}]$) which represents an initial segment of length $\alpha$ of a potentially transfinite construction. The computation proceeds by handing this partial object to $U$, together with a recursive continuation describing how the approximation can be extended beyond $\alpha,$ by applying $G$ to the space of all possible extensions of $t\restr \alpha$ (i.e. $\Ind[G, U, t, \underline{\alpha}]$).

%If so, $U$ produces the corresponding output; otherwise the computation recursively continues.}

%Then, the reductum is the application to $U$ of (a representative for) the sequence $(u_\xi)_{\xi<\kappa}$, where $u_\xi=(G)\Ind[G, U, t, \underline{\alpha}]$ for $\xi\geq\alpha$. The term $(G)\Ind[G, U, t, \underline{\alpha}]$ will be employed to find a term $u$ to properly extend the sequence $(u_\xi)_{\xi<\alpha}$, as shown in Lemma \ref{lem:succstep}.

We show that 
%under certain assumptions on $G$ and $U$, we have $\BR GU\underline{0}\,\underline{0}\in \falsitydual{\bot},$ 
by means of transfinite bar-induction and $\kappa$-\closure, the realizability algebra does not add new $<\kappa$-sequences (Theorem \ref{thm:BRtermination}), this is the key theorem that will imply the main results of this paper.
%It is not wrong to informally think about the condition $\BR GU\underline{0}t_0\Vdash\bot$ as $\BR$ terminates on $G, U$.

We must first impose some conditions on our choice of the terms $\underline{\alpha}.$

\begin{definition}\label{def:numerals} Let $\mathcal{K}=\langle\Lambda_\mathsf{c},\Pi,\succ,\Perp\rangle$ be a realizability algebra and $\kappa\leq |\mathcal{K}|$. We call \emph{ordinal terms} a sequence of terms $(\underline{\alpha})_{\alpha<\kappa}\subseteq\Lambda_\mathsf{c}$ satisfying the following condition:
%\begin{enumerate}
    %\item $\underline{\alpha+1}=(\underline{s}) \underline{\alpha}$ for some term $\underline{s}$
    %\item there exists a sequence of terms $(o_\alpha)_{\alpha<\kappa}$ such that, for every $\xi<\alpha<\kappa,$ $(\underline{\alpha})\underline{\xi}\dsucc o_\xi$
    \item 
    
    for every $\delta<\kappa$ limit and every $v\in \Lambda_c$ with $FV(v)\subseteq \{x\},$ if $v\{x:=\underline{\delta}\}\in \falsitydual{\bot}$ then there is $\alpha<\delta$ such that $v\{x:= \underline{\alpha+1}\}\in \falsitydual{\bot}$  
%\end{enumerate}
\end{definition}

Hat ordinals are defined accordingly, using ordinal terms for the sequence $(\underline{\alpha})_{\alpha<\kappa}$ in Definition \ref{def: hat ordinals}. 
%The ordinal terms $(\underline{\alpha})_{\alpha<\kappa}$ can be thought as generalizations of numerals.  
%satisfying the additional property of \emph{accumulating} some sequence $(u_\alpha)_{\alpha<\kappa}$ ; 
In \cref{subsec: exemple} we will give an example of non trivial realizability algebra that has a system of ordinal terms 
%that accumulate the sequence $(I)_{\alpha<\kappa}$ (where $I$ denotes the identity $\lambda x.x$). 
%So that 
where each $\underline{\alpha}$ accumulates the sequence $(I)_{\beta<\alpha}$ (where $I$ denotes the identity $\lambda x.x$). 
%Note that the sequence $(u_\alpha)_{\alpha<\kappa}$ is fixed for all $\underline{\alpha}$.  

\begin{definition}\label{def: canonical system}
Let $\mathcal{K}$ a realizability algebra  
with a system of ordinal terms $(\underline{\alpha})_{\alpha<\kappa}$ for some regular cardinal $\kappa,$ and containing a trichotomy operator $\chi$ with respect to such ordinal terms. We say that $\mathcal{K}$ has an \emph{inductive system of terms} for $\kappa$ if for every limit $\gamma<\kappa,$ we can associate to every $\leq \delta$-sequence of terms $s:= (u_\xi)_{\xi<\delta}$ with $\delta< \gamma,$ an accumulator $t_s$ satisfying the following properties:    
\begin{enumerate}
        \item if $s$ is the empty sequence, then $t_s$ is $\underline{0}$ (any term is an accumulator of the empty sequence)
        \item if $\delta=\alpha+1$, then  $t_s\dfn (\chi)\underline{\alpha} t_{s\restr \alpha} u_\alpha$ 
        %where $t$ is a fixed minimal accumulator for the sequence $(u_\xi)_{\xi<\alpha}$ given by the $\kappa$-\closure. 
        \item if $\delta$ is a limit ordinal $<\gamma$, then $t_s$ is an accumulator for the sequence $(u_\xi )_{\xi<\delta}$
        %$$(\Restr[t_{s\restr {\beta+1}}, \underline{\beta+1}](G)\Ind[G, U, \underline{\beta+1}, t_{s\restr {\beta+1}}])_{\beta<\alpha}$$
        such that \emph{for every} $\alpha<\delta$ and for every $v\in \Lambda_c$ with $FV(v)\subseteq \{x\},$ if $v\{x:= \Restr(t_s, \underline{\alpha})\}\in \falsitydual{\bot}$ then $v\{x:= \Restr(t_{s\restr \alpha}, \underline{\alpha})\}\in\falsitydual{\bot}.$
        %\item if $\delta=\gamma,$ then $t_s$ is a minimal accumulator for $(t_{s\restr \alpha})_{\alpha<\delta}.$
%Note that in particular $t_s$ is an accumulator for the sequence $(u_\xi)_{\xi<\delta}.$    
\end{enumerate}
Every $t_s$ will be called the \emph{canonical accumulator of} $s$.
\end{definition}

Roughly, item $3$ of the definition above means that for every $\alpha<\delta,$ when considering $t_s$ up to $\alpha,$ it computes exactly the same terms as $t_{s\restr \alpha}.$

In Section \ref{subsec: exemple}, we will give an example of a non trivial $\kappa$-\closed realizability algebra satisfying this condition. 

Now we state and prove our central theorem that will be used to show that $\kappa$-\closed realizability algebras with an inductive system of terms induce realizability models for fragments of choice and preserve cardinals up to $\kappa.$ Intuitively, this theorem says that a realizability algebra satisfying such conditions will not add new $<\kappa$-sequences. 

\begin{theorem}\label{thm:BRtermination}
    Consider an (uncountable) realizability algebra $\mathcal{K}$ such that, for a regular uncountable cardinal $\kappa\leq |\mathcal{K}|$, 
    %there is a sequence of ordinal terms $(\underline{\alpha})_{\alpha<\kappa}\subset \Lambda_\mathsf{c}$, a trichotomy operator $\chi\in\Lambda_\mathsf{c}$ and 
    $\mathcal{K}$ is $\kappa$-\closed 
    and there is an inductive system of terms for $\kappa.$
    Let $\gamma<\kappa$ be a limit ordinal and suppose that for some formula $F(\vec{c}, x, y)$ with parameters $\vec{c}\in M^\mathcal{K},$
    %let $F(\vec{w}, x, y)$ be a formula of $\mathrm{ZF}_{\rlzin}$. Suppose that for some $\vec{c}\in M^\mathcal{K}$, 
    there are terms $G,U$ such that $G\in\falsitydual{\forall x\exists y F[\vec{c}, x,y]}$ and $U\in\bigcap_{f:\gamma\to M^\mathcal{K}}\falsitydual{\neg\forall x^{\hat{\gamma}}F[\vec{c}, x, \tilde{f}(x)]}$. Then, $\BR GU\underline{0}\,\underline{0}\in \falsitydual{\bot}$.
\end{theorem}
\begin{proof}
    To simplify the notation and ease readability of the proof, we omit the parameters $\vec{c}.$ 
    We want to apply transfinite bar-induction to a tree of $<\gamma$-sequences of $\lambda_c$-terms, for that we 
    use the inductive system of terms; for every sequence $s,$ we denote $t_s$ its canonical accumulator in the inductive system.
    
    %fix for every $\leq \gamma$-sequence of terms $s:= (u_\xi)_{\xi<\alpha}$ (i.e. with $\alpha\leq \gamma$), a convenient accumulator $t_s$ inductively defined as follows: 
    %\begin{itemize}
    %    \item if $s$ is the empty sequence, then $t_s$ is $\underline{0}$ (any term is an accumulator of the empty sequence)
    %    \item if $\alpha=\beta+1$, then let $t_s\dfn (\chi)\underline{\beta} t u_\beta$ where  
    %    $t$ is a fixed minimal accumulator for the sequence $(u_\xi)_{\xi<\beta}$ given by the $\kappa$-\closure. 
    %    \item if $\alpha$ is a limit ordinal, we let $t_s$ be a \todo{minimal accumulator for the sequence $(t_{s\restr {\beta}} )_{\beta<\alpha}$
        
%Note that in particular $t_s$ is an accumulator for the sequence $(u_\xi)_{\xi<\beta}.$    
%}
%\end{itemize}

    %Given a canonical accumulator $t_s$ for a sequence $s=(u_\xi)_{\xi<\alpha}$, we denote by $lg(t)$ the length of the sequence (namely $\alpha$). 
    Let $H=\BR GU$. We define two predicates $T, B$ on $\Lambda_\mathsf{c}^{<\gamma}$ as follows: for $s:= (u_\xi)_{\xi<\alpha},$
    \[T(s)\dfn \forall\xi<\alpha \exists d\in M^{\mathcal{K}}(u_\xi\in \falsitydual{F[\hat{\xi}, d]}),\]
    \[B(s)\dfn T(s)\ \land\  
    (U)(\Restr[t_s, \underline{\alpha}])(G)\Ind[G, U, t_s, \underline{\alpha}]
    %H\underline{\alpha}t_{s}
    \in \falsitydual{\bot},\]

%when $\alpha$ is a limit ordinal 

%\[B(s)\dfn T(s)\ \land\  Ut_s\in \falsitydual{\bot} 
 %   \]

    It is immediate to see that $T$ is a tree on $\Lambda_\mathsf{c}$ and that it is closed by $\sqsubseteq$-increasing sequences of length $<\gamma.$ Next lemmata show that $B$ and $T$ satisfy the hypothesis of transfinite bar-induction on $\gamma$, namely that 
    %$T$ is closed by $\sqsubseteq$-increasing sequences of length $<\kappa,$ 
    $B$ is a bar for $T$ (Lemma \ref{lem:Binductive}) and that $B$ is inductive (Lemma \ref{lem:succstep}, Lemma \ref{lem:Blimitcase}). Then, by transfinite bar-induction, we can conclude that 
    $(U)(\Restr[\underline{0}, \underline{0}])(G)\Ind[G, U, \underline{0}, \underline{0}]\in \falsitydual{\bot},$ holds as it corresponds to $B(\langle\,\rangle).$ 
Therefore 
    $\BR GU\underline{0}\,\underline{0}\in \falsitydual{\bot}$ since it reduces to the above term. 
    
%In the following, $H$ represents the term $\BR GU$ for $G,U$ such as in the hypothesis of \cref{thm:BRtermination}, and $T, B$ are the predicates defined in the above proof. 

\begin{lemma}\label{lem:Binductive}
    For every $\gamma$-branch  
    $s:=(u_\xi)_{\xi<\gamma}$ of $T,$ there is $\delta<\gamma$ such that $B(s\restr \delta)$
%    $H\underline{\delta}t_{s\restr \delta}\in \falsitydual{\bot}.$      
\end{lemma}
\begin{proof}
    By hypothesis and using $\mathrm{AC}$ in $\mathcal{M}$ (or just $\mathrm{AC}_\gamma$), we can fix a choice function  
    %using $\mathrm{DC}_\gamma$, we can define a function 
    $f:\gamma\to M^\mathcal{K}$ 
    such that for every $\xi<\gamma,$ $u_\xi\in \falsitydual{F[\hat{\xi}, f(\xi)]}.$
By $\kappa$-\closure there is a minimal accumulator $t_s$ for the sequence $(u_\xi)_{\xi<\gamma}.$ Since $t_s$ is an accumulator of the sequence,  
we have $t_s\in\falsitydual{\forall x^{\hat{\gamma}}F[x,\tilde{f}(x)]}$, where $\tilde{f}\in M^\mathcal{K}$ is the lift associated to $f$. 
    Since $U\in\falsitydual{\neg\forall x^{\hat{\gamma}} F[x,\tilde{f}(x)]}$, we have $(U)t_s\in \falsitydual{\bot}$. 
    Since $t_s$ is a minimal accumulator of the sequence $s,$ there exists $\delta<\gamma$ such that for every accumulator $t'$ of the sequence $(u_\xi)_{\xi<\delta},$ we have $Ut'\in\falsitydual{\bot}.$
    %\todo{{\bf Case 1:} $\delta$ is a successor or $0.$} 
    Then consider 
    $$t':= \Restr[t_{s\restr \delta}, \underline{\delta}](G)\Ind[G,U,\underline{\delta},t_{s\restr \delta}],$$ it accumulates the sequence up to $\delta,$ hence $(U)t'\in \falsitydual{\bot}$ that is 
    $$
    %H\underline{\delta}t_{s\restr \delta}\dsucc
    (U)\Restr[t_{s\restr \delta}, \underline{\delta}](G)\Ind[G,U,\underline{\delta},t_{s\restr \delta}]=(U)t'\in \falsitydual{\bot}$$ It follows that %H\underline{\delta}t_{s\restr \delta}\in \falsitydual{\bot},$ hence 
    $B(s\restr \delta).$
%    \todo{{\bf Case 2:} $\delta$ is limit and $>0.$ Then $t_{s\restr \delta}$ is also an accumulator of the sequence $s$ up to $\delta,$ therefore  $Ut_{s\restr \delta}\in \falsitydual{\bot},$ hence $B(s\restr \delta).$ }     
\end{proof}

\begin{lemma}\label{lem:succstep}
    Given a sequence $s:=(u_\xi)_{\xi<\alpha}$ for $\alpha<\gamma,$ suppose that $T(s)$ and for all $u\in\Lambda_\mathsf{c}$, if $T(s\smallfrown u)$ holds then $B(s\smallfrown u)$. Then, $B(s)$ holds.
\end{lemma}
\begin{proof}
    We prove the contrapositive. Suppose $T(s)$ and $\lnot B(s),$
    %$H\underline{\alpha}t_s\not\in \falsitydual{\bot},$ 
    we are going to find $u\in \Lambda_c$ such that $T(s\smallfrown u) \land \lnot B(s\smallfrown u).$ 
    Since $T(s)$ holds, using AC in $\mathcal{M}$ (or just $AC_\gamma$), we can choose 
    %by $\mathrm{DC}_\gamma$ we can fix 
    for every $\xi<\gamma,$ some values $d_\xi$ such that $u_\xi\in \falsitydual{F[\hat{\xi}, d_\xi]}.$ 
    To further simplify the notation, let us write $\Ind$ for $\Ind[G, U, t_s, \underline{\alpha}].$\\ 
    %so that $H\underline{\alpha}t \dsucc (U)(\Restr[t_s, \underline{\alpha}])(G)\Ind.$\\ 
    
    {\bf Claim:} 
    $\Ind\not\in\falsitydual{\forall y(F[\hat{\alpha},y]\rightarrow\bot)}$.\\ 
     
\emph{Proof of claim.} Suppose not, then since $G\in\falsitydual{\forall x\exists y F(x,y)}$, unfolding the definition of the falsity value we have $$G\in\falsitydual{\exists yF[\hat{\alpha}, y]}=
    \falsitydual{\forall y(F[\hat{\alpha},y]\rightarrow\bot) \rightarrow\bot},$$ and in particular $(G)\Ind\in \falsitydual{\bot}$. 
    %\todo{We distinguish two cases.\\ 
    %{\bf Case 1:} $\alpha$ is a successor ordinal or $0.$
    %}
    Then, setting $L\dfn (\Restr[t_s, \underline{\alpha}])(G)\Ind$, for $\xi<\alpha,$ we would have $L \underline{\xi}\dsucc (t_s)\underline{\xi}\dsucc u_\xi\in \falsitydual{F[\hat{\xi}, d_\xi]}$, while for $\xi\geq \alpha,$ $L\underline{\xi}\dsucc (G)\Ind\in \falsitydual{\bot}$. In particular, the latter case implies $(G)\Ind\in \falsitydual{F[\hat{\xi}, \emptyset]}$. Define $f(\hat{\xi})=d_\xi$ if $\xi<\alpha$, and $f(\hat{\xi})=\emptyset$ otherwise. By construction $L\in\falsitydual{\forall x^{\hat{\gamma}}F(x,\tilde{f}(x))}.$ 
    By hypothesis on $U$ we have in particular $U\in\falsitydual{\neg\forall x^{\hat{\gamma}} F(x,\tilde{f}(x))}$ by hypothesis. 
    Hence we have $UL\in \falsitydual{\bot},$
    that is 
    $
    %H\underline{\alpha}t_s\dsucc
    (U)(\Restr[t_s, \underline{\alpha}])(G)\Ind=
    UL\in\falsitydual{\bot}$, contradicting %$H\underline{\alpha}t_s\notin \falsitydual{\bot}$.%\\ 
    $\lnot B(s).$
    %\todo{ {\bf Case 2:} $\alpha$ is a limit ordinal Then $t_s \underline{\xi}\dsucc u_\xi\in \falsitydual{F[\hat{\xi}, d_\xi]}$, while for $\xi\geq \alpha,$ ? 
    %} 
        
    \emph{End proof of claim.}\\

    $\Ind\not\in\falsitydual{\forall y(F[\hat{\alpha},y]\rightarrow\bot)}$ implies that there is a name $d\in M^\mathcal{K}$ and a term $u\in\falsitydual{F[\hat{\alpha}, d]}$ such that $(\Ind)u\notin \falsitydual{\bot}$. 
    On the other hand, 
    $$(\Ind)u\succ H\underline{\alpha+1}t_{s\smallfrown u}\dsucc 
    (U)(\Restr[t_{s\smallfrown u}, \underline{\alpha+1}])(G)\Ind[G, U, t_{s\smallfrown u},\underline{\alpha+1}].
    $$ 
    Therefore we must have $
    %H\underline{\alpha+1}t_{s\smallfrown u}
    (U)(\Restr[t_{s\smallfrown u}, \underline{\alpha+1}])(G)\Ind[G, U, t_{s\smallfrown u},\underline{\alpha+1}]
    \notin \falsitydual{\bot},$
    from which the result follows.
\end{proof}

\begin{lemma}\label{lem:Blimitcase}
    Suppose that for $\delta<\gamma,$ $(s_\alpha)_{\alpha<\delta}$ is a $\sqsubseteq$-increasing sequence such that $T(\bigcup_{\alpha<\delta} s_\alpha)$ and $B(\bigcup_{\alpha<\delta} s_\alpha)$. Then, there exists $\alpha<\delta$ such that $B(s_\alpha)$.
\end{lemma}
\begin{proof}
    Notice that if $\delta$ is a successor ordinal equal to $\beta+1,$ then $\bigcup_{\alpha<\delta} s_\alpha=s_\beta,$ thus $B(s_\beta)$ by assumption. The only interesting case is when $\delta$ is limit. Since $(s_\alpha)_{\alpha<\delta}$ is $\sqsubseteq$-increasing and $T$ is closed by sequences of length $<\gamma,$ we have that  $\sigma:=\bigcup_{\alpha<\delta} s_\alpha$ satisfies $T(\sigma).$ 

%\todo{
%Suppose that 
%$B(\sigma),$ 
%that is $Ut_{\sigma}\in \falsitydual{\bot}.$ 
%    By $\kappa$-\closure and by definition of $t_{\sigma}$, 
%    there is $\alpha<\delta$ such that  
%    $(U)w\in \falsitydual{\bot}$ for every accumulator $w$ of the sequence 
%        $$(\Restr[t_{\sigma\restr\beta}, \underline{\beta}](G)\Ind[G, U, \underline{\beta}, t_{\sigma\restr \beta}])_{\beta<\alpha}$$
%    If $\alpha$ is a limit ordinal, then in particular $(U)t_{s\restr \alpha}\in \falsitydual{\bot},$ hence $B(\sigma\restr \alpha).$ 
%    If $\alpha$ is a successor, then in particular $H\underline{\alpha}t\restr\alpha\dsucc (U)\Restr[t_{\sigma\restr\alpha}, \underline{\alpha}](G)\Ind[G, U, \underline{\alpha}, t_{\sigma\restr \alpha}]$ hence $B(\sigma\restr \alpha).$
% }   
% \end{comment}
    
    To simplify the notation, for every $\alpha\leq \delta$ denote $r_\alpha$ the term $t_{\sigma\restr \alpha}$ which is also $t_{s_\alpha}$ since the sequence is increasing.  
    Suppose $B(\sigma).$ that is 
    %$H\underline{\delta}t_\delta$ with 
    $(U)\Restr[r_{\delta}, \underline{\delta}](G)\Ind[G,U,\underline{\delta},r_{\delta}]\in \falsitydual{\bot}.$
For every $\alpha$ and $r,$ define
$$\varphi[\underline{\alpha}, r]\dfn \Restr[r, \underline{\alpha}](G)\Ind[G,U,\underline{\alpha},r].$$

Let $v \dfn \lambda x,y (U)\varphi[x, y].$ 
Since 
$(U)\varphi[\underline{\delta}, 
r_{\delta}]\in \falsitydual{\bot},$
we have $v\underline{\delta}r_\delta\in \falsitydual{\bot}.$ 
By Definition \ref{def:numerals},  
there is $\alpha<\delta$ such that 
%for every accumulator $w$ of the sequence 
%$(o_{\xi})_{\xi<\alpha},$ 
%we have $(U)\varphi[w,t_\delta]\in\falsitydual{\bot,}$
%In particular for $\underline{\alpha},$ we have
$(U)\varphi[\underline{\alpha+1},r_\delta]\in\falsitydual{\bot}.$ 
By item 3. of the definition of an inductive system of terms, we must have 
$(U)\varphi[\underline{\alpha+1},r_{\alpha+1}]\in \falsitydual{\bot},$ 
hence $B(s_{\alpha+1}).$

\begin{comment}    
    Define a term $w$ such that 
     \[   
           \quad w\underline{\beta}\star \pi\dsucc
                \begin{cases}
                    (t_\alpha)\underline{\beta}\star\pi&\textrm{if }\beta<\alpha;\\
                    (G)Rec[G, U, \underline{\alpha}, t_\alpha]\star\pi&\textrm{if }\beta\geq\alpha.
                \end{cases}
        \]
    
    Then $w$ is an accumulator of $(t_\xi)_{\xi<\alpha}$ hence 
    $(U)\varphi[\underline{\delta}, w]\in \falsitydual{\bot}.$
\end{comment}

\end{proof}
That completes the proof of \cref{thm:BRtermination}. 
\end{proof}

%%%

\section{Choice}\label{sec: ACkappa}

We now show that using transfinite bar-recursion and $\kappa$-\closure, it is possible to realize the Axiom of Choice restricted to families of sets indexed by ordinals in $\kappa;$ we will apply \cref{thm:BRtermination}. 

%\begin{definition}
%In $ZF_\varepsilon,$ let $\gamma$ be an ordinal, denote $(DC_\varepsilon)_\gamma$ the scheme of axiom that for every formula $F(x, y, \vec{a})$ if $\forall x\exists yF(x,y, \vec{a}),$ then there is a function $f$ such that $\forall \alpha\rlzin {\gamma} F(f\restr \alpha,f(\alpha)))$ 
%\end{definition}

%\begin{definition}
%Let $\gamma$ be an ordinal, denote $(AC_\varepsilon)_\gamma$ the following scheme of axiom: for every formula $F(x, y, \vec{a})$  in $\mathscr{F}_{\rlzin},$ if $\forall x\exists yF(x,y, \vec{a}),$ then there is an extensional function $f$ such that $\forall \alpha\rlzin {\gamma} F( \alpha,f(\alpha)), \vec{a})$ 
%\end{definition}

\begin{definition}
    For an ordinal $\gamma$ in $\mathrm{ZF}_\varepsilon$ and a formula $F$ in $\mathscr{F}_{\rlzin},$ 
    we denote $\mathrm{AC}_{\gamma, F}$ the following: 
        \[\mathrm{AC}_{\gamma, F}\dfn
        \forall\vec{w}(\forall x\exists yF(\vec{w},x,y)\rightarrow\exists f( Fun(f)\land\forall x\rlzin \gamma F(\vec{w},x,f(x)))).\]
     
     %For an ordinal $\kappa<|\mathcal{K}|$, we denote  $\mathrm{AC}_{<\hat{\kappa}}$ the schema of axiom $\forall \gamma\rlzin\hat{\kappa}\ \mathrm{AC}_{\gamma}$. 
\end{definition}

\begin{remark}\label{rmk: AC hat gamma implies AC gamma} Observe that for every $\gamma\rlzin\hat{\kappa},$ if $\mathrm{AC}_{\gamma, F}$ holds for every $F$ in $\mathscr{F}_{\rlzin},$ then $AC_{\gamma}$ 
holds, namely every family of $\gamma$ non empty sets has an \emph{extensional} choice function: indeed 
if $\mathcal{E}=\{E_\alpha\}_{\alpha<\gamma}$ is such a family and $F(x, y)$ is the formula $x\varepsilon \gamma\imp y\in E_x,$ then by 
$\mathrm{AC}_{\gamma, F},$ 
there is a function $f$ on $\gamma$ such that for every $\alpha\rlzin \gamma,\ F(\alpha, f(\alpha)),$ namely $f(\alpha)\in E_\alpha.$ By Lemma \ref{lem: trichotomy operator} (using $\chi$), $\hat{\kappa}$ hence $\gamma$ are  $\varepsilon$-TODs, and by Lemma \ref{lem: TODs are good}, for every $\varepsilon$-TOD $a,$ we have $\forall x, y\rlzin a(x\simeq y \imp x = y),$ hence every function on $a$ is extensional. In particular $f$ is an extensional function on $\hat{\gamma}.$    
\end{remark}

\begin{theorem}\label{thm:ACk} Let $\mathcal{K}$ as in the hypotheses of Theorem \ref{thm:BRtermination}. 
    %Consider a realizability algebra $\mathcal{K}$ such that, for some regular uncountable cardinal $\kappa\leq |\mathcal{K}|$, there is a sequence of ordinal terms $(\underline{\alpha})_{\alpha<\kappa}\subset \Lambda_\mathsf{c}$, a trichotomy operator $\chi\in\Lambda_\mathsf{c}$ and $\mathcal{K}$ is $\kappa$-\closed. 
    Then $\mathcal{N}\models \forall \alpha<\hat{\kappa}\ \mathrm{AC}_{\alpha}$.
\end{theorem}
\begin{proof} We are going to prove that for every limit ordinal $\gamma$ below $\kappa,$ and for every formula 
$F\in \mathscr{F}_{\rlzin},$ 
$AC_{\hat{\gamma}, F}$ is realized by a term that does not depend on $\gamma,$ and we will conclude that $AC_\alpha$ holds for every ordinal $\alpha$ below $\kappa.$

    In $\mathcal{M}$, let $L$ be the set of limit ordinals strictly below $\kappa.$ Define $\hat{L}\dfn \{\hat{\gamma};\ \gamma\in L\}$ and consider $\gimel \hat{L}.$

    \begin{lemma} $\Vdash \gimel \hat{L}\subsetneq \hat{\kappa}$
    \end{lemma}
    
    \begin{proof} By Lemma \ref{lem: hat ordinals are epsilon ordinals} we can fix a proof-like term $\theta$ such that for all $\hat{\alpha},$ $\theta$ realizes that $\alpha$ is an $\rlzin$-ordinals. Then in particular $\theta\Vdash \forall x^{\gimel \hat{L}}(x \textrm{ is an $\rlzin$-ordinal}),$ hence $\Vdash \gimel \hat{L}\subseteq Ord.$ Clearly $\falsity{\hat{\kappa}\notrlzin \gimel \hat{L}}=\emptyset,$ hence $\Vdash \hat{\kappa}\notrlzin \gimel \hat{L}.$ By Lemma \ref{lem: inclusione hat ordinals}, we can fix a proof-like term $\theta'$ that realizes $\hat{\gamma}\subseteq \hat{\delta}$ for every ordinal $\gamma < \delta.$ In particular, $\theta'\Vdash \forall x^{\gimel\hat{L}} (x\subseteq \hat{\kappa}).$ It follows that $\Vdash \forall x \rlzin \gimel\hat{L} (Ord(x) \land x< \hat{\kappa}),$ this completes the proof of the lemma.
    \end{proof}

    \begin{lemma}\label{lem: hat limits are cofinal} 
    $\Vdash \forall x^{\hat{\kappa}}\exists y^{\gimel \hat{L}} (x\subseteq y)$
    \end{lemma}

    \begin{proof} As above, we fix a proof-like term $\theta'$ that realizes $\hat{\gamma}\subseteq \hat{\delta}$ for every ordinal $\gamma < \delta$ (Lemma \ref{lem: inclusione hat ordinals});  we prove that $\lambda x.x\theta' \Vdash \forall x^{\hat{\kappa}}((\forall y^{\gimel \hat{L}} x\not \subseteq y)\imp \bot).$ Fix $\alpha<\kappa,$ and consider $t\in \falsitydual{\forall y^{\gimel \hat{L}} (\hat{\alpha} \not\subseteq y)}.$ Let $\gamma\dfn \alpha+\omega,$ then $\gamma$ is a limit ordinal below $\kappa$ (since $\kappa$ is regular uncountable), hence $\hat{\gamma}\in \hat{L}.$  Since $\alpha<\gamma,$ we have $\theta'\Vdash \hat\alpha \subseteq \hat{\gamma},$ hence $t \theta' \in \falsitydual{\bot},$ therefore $\lambda x.x\theta' \star t.\pi \in \Perp$ for every $\pi.$ \end{proof}
    
\begin{lemma}\label{lem: AC for hat limits}
    For every formula $F(\vec{w}, x, y)\in \mathscr{F}_{\rlzin},$ we have 
    \begin{equation}\label{eq:ACgamma}
    %\lambda gu.\BR gu\underline{0}\,\underline{0} 
    \Vdash \forall x^{\gimel\hat{L}}\mathrm{AC}_{x, F}
    %\forall\vec{w}(\forall \vec{x}\exists yF(\vec{w},x,y)\rightarrow\exists f\forall x\rlzin \hat{\gamma} F(\vec{w},x,f(x))).
    \end{equation}
    \end{lemma}
    \begin{proof}    
    Fix $\gamma<\kappa$ limit (in $\mathcal{M}$). To simplify the notation we omit the parameters $\vec{w}$ in $F$. The formula in \cref{eq:ACgamma} is equivalent to the following: 
    $$\forall x\exists yF(x,y)\rightarrow\forall f\neg(Fun(f)\land\forall x^{\hat{\gamma}} F(x,f(x)))\rightarrow\bot,$$ this is the formula we are going to realize. Let $\theta_1$ be the proof-like term given by Proposition \ref{prop:functionlift}, we show that  
    $\lambda gu.(\BR g(\lambda f.u(\lambda x.x\theta_1f))\underline{0}\,\underline{0}$ realizes the formula above. 
    Any stack in its falsity value is in the form $G\centerdot U\centerdot\pi$, for $G\in\falsitydual{\forall x\exists yF(x,y)}$, $U\in\falsitydual{\forall f\neg(Fun(f)\land\forall x^{\hat{\gamma}} F(x,f(x)))}$, and $\pi\in\falsity{\bot}$. We first show that $\lambda f.U(\lambda x.x\theta_1f)\in\bigcap_{f:\gamma\to M^{\mathcal{K}}} \falsitydual{\forall x^{\hat{\gamma}} F(x,\tilde{f}(x))\rightarrow\bot}$, where $\theta_1$ is the term given by Proposition \ref{prop:functionlift}. For, fix $f:\gamma\to M^{\mathcal{K}}$ and consider $\rho\in\falsity{\forall x^{\hat{\gamma}} F(x,\tilde{f}(x))\rightarrow\bot}$, which has the form $t\centerdot\rho'$ for $t\in\falsitydual{\forall x^{\hat{\gamma}} F(x,\tilde{f}(x))}$, $\rho'\in\Pi$. Then, $\lambda x.x\theta_1t\in\falsitydual{Fun(\tilde{f})\land\forall x^{\hat{\gamma}} F(x,\tilde{f}(x))}$ and $U\star\lambda x.x\theta_1t\centerdot\rho'\in\Perp$ by hypothesis on $U$. Therefore, $G$ and $U'\dfn \lambda f.U(\lambda x.x\theta_1f)$ fit the hypothesis of Theorem \ref{thm:BRtermination}, hence $\BR GU'\underline{0}\,\underline{0}\in \falsitydual{\bot},$ which entails $\BR G U'\underline{0}\,\underline{0}\star \pi\in \Perp.$ Therefore,   
        $(\lambda gu.(\BR g(\lambda f.u(\lambda x.x\theta_1f))\underline{0}\,\underline{0})\star G\centerdot U\centerdot \pi \in \Perp.$  
    \end{proof}
    
    Now we can conclude that $\mathcal{N}\models \forall \alpha < \hat{\kappa}\,\mathrm{AC}_\alpha:$ indeed in $\mathcal{N}$ let $\alpha$ an ordinal below $\kappa,$ then by Lemma \ref{lem: hat limits are cofinal} there is $\gamma\in \gimel \hat{L}$ such that $\alpha\leq \gamma.$ By Lemma \ref{lem: AC for hat limits}, $\mathrm{AC}_{\gamma, F}$ holds for every formula $F,$ which obviously implies $\mathrm{AC}_{\alpha, F}.$ By Remark \ref{rmk: AC hat gamma implies AC gamma}, $\mathrm{AC}_\alpha$ holds.     
\end{proof}

%%%

\section{Preservation of cardinals}\label{sec: cardinal preservation}

We proved that for every ordinal $\kappa$ in the ground model a $\kappa$-\closed realizability algebra with an inductive system of terms induces a realizablity model for $<\hat{\kappa}$-choice; in particular, if $\kappa$ is large enough, we can realize possibly uncountable fragments of choice.  However, as explained in \cref{sec:closure}, $\hat{\kappa}$ may collapse, for instance, to $\omega_1,$ implying that only countable choice is guaranteed to hold in $\mathcal{N}.$ For this reason, we show that for every cardinal $\mu\leq\kappa$, its representative $\hat{\mu}$ remains a cardinal in the realizability model, by showing that there exists no surjection from an ordinal $\gamma<\hat{\mu}$ onto $\hat{\mu}.$ Since extensional equality and the strict identity coincide on hat ordinals (Lemma \ref{lem: TODs are good}), every function whose domain and range are hat ordinals is also an $\rlzin$-function, hence it will be enough to show that there is no $\rlzin$-function which is a surjection from an ordinal of $\hat{\mu}$ onto $\hat{\mu}.$ 

Given a name $f$, and two ordinals $\mu, \gamma<\kappa$, define the following formulas: 
\begin{enumerate}
    \item $Tot(f, \hat{\gamma})\equiv \forall x\exists y\ (x\rlzin \hat{\gamma}\imp \op (x, y)\rlzin f)$;
    %\item $Tot(f)\equiv \forall x\exists y\ \op (x, y)\in f$;
    \item $Fun(f, \hat{\gamma})\equiv f$ is  an $\rlzin$-function on $\hat{\gamma};$
    %\forall x^{\hat{\gamma}}\forall y\forall y' \op(x, y)\varepsilon f, \op(x, y')\varepsilon f\imp y=y'$;
    \item $Surj(f, \hat{\gamma}, \hat{\mu})\equiv \forall y^{\hat{\mu}}\exists x^{\hat{\gamma}}\ \op(x, y)\varepsilon f$;
    
\end{enumerate}
The formula $Tot(f, \hat{\gamma}, \hat{\mu})$ is satisfied when $f$ is a total binary relation with domain $\hat{\gamma};$ 
%$Fun(f)$ indicates that $f$ is an $\rlzin$-function on $\hat{\gamma},$ and   
$Surj(f, \hat{\gamma}, \hat{\mu})$ expresses that the image of $f$ is $\hat{\mu}$ (when $f$ is a function from $\hat{\gamma}$ to $\hat{\mu}$, the formula ensures that $f$ is surjective). 

\begin{lemma}\label{lem:luglio}
    Given a function $f: \gamma\to\mu$ in the ground model for $\gamma,\mu<\kappa$, if $f$ is not a surjection, then $\lambda x.xI\Vdash \neg Surj(\tilde{f},\hat{\gamma},\hat{\mu})$;    
\end{lemma}
\begin{proof}
    Fix $\beta<\mu$ such that $\beta\not\in\im{f}$. Then, for every $\alpha<\gamma$, the falsity value $\falsity{\op(\hat{\alpha},\hat{\beta})\notrlzin\tilde{f}}$ is equal to the empty set. Thus, $\falsity{\forall x^{\hat{\gamma}}(\op(x,\hat{\beta})\notrlzin\tilde{f})}=\emptyset$, which implies that any term realizes the formula (in particular $I$). Given $t\Vdash Surj(\tilde{f},\hat{\gamma}, \hat{\mu})$, the term $tI\Vdash\bot$, thus $\lambda x.xI\Vdash\neg Surj(\tilde{f},\hat{\gamma}, \hat{\mu})$.
\end{proof}

\begin{theorem}\label{thm:preserving cardinals}
Let $\mathcal{K}$ as in the hypotheses of Theorem \ref{thm:BRtermination}. 
    %Consider a realizability algebra $\mathcal{K}$ such that, for a fixed cardinal $\kappa\leq |\mathcal{K}|$, there exists a sequence of ordinal terms $(\underline{\alpha})_{\alpha<\kappa}\subset \Lambda_\mathsf{c}$, there exists a trichotomy operator $\chi\in\Lambda_\mathsf{c}$ and $\mathcal{K}$ is $\kappa$-\closed. 
    Then for every cardinal $\mu\leq \kappa,$ we have
    $\mathcal{N}\models \hat{\mu}$ is a cardinal.
\end{theorem}
\begin{proof}
    Recall that an ordinal $\mu$ is a cardinal when there is no surjection from $\gamma<\mu$ onto $\mu.$
    
    Fix $\mu\leq \kappa,$ we are going to prove that for all $\gamma<\mu$ limit, 
    \begin{equation}\label{eq: surjection}
    \mathcal{N}\models \forall f((Tot(f, \hat{\gamma})\land  \forall g(Fun(g, 
    \hat{\gamma}),\forall x(x\rlzin \hat{\gamma}\imp \op(x, g(x))\rlzin f)\rightarrow\ Surj(g, \hat{\gamma}, \hat{\mu})))\rightarrow\bot).
    \end{equation}
    This formula implies that there is no surjection $f$ from $\hat{\gamma}$ onto $\hat{\mu}$ in the realizability model $\mathcal{N}$. Indeed, if by contradiction such surjection exists in $\mathcal{N}$, then $Tot(f, \hat{\gamma})$ holds. 
    %and for all $\beta\rlzin\hat{\mu}$, the set $f_\beta^{-1} =\{\alpha\rlzin\hat{\gamma}\mid \op(\alpha,\beta)\rlzin f\}$ is non-empty. 
    Moreover, for every $\rlzin$-function $g$ on $\hat{\gamma}$ such that $\forall x\rlzin \hat{\gamma} (\op(x, g(x))\rlzin f),$ we have that $g$ is surjective onto $\hat{\mu}$, because for every $\beta\rlzin\hat{\mu},$ there is $\alpha\rlzin\hat{\gamma}$ such that $f(\alpha)=\beta$ ($f$ is surjective), and  $g(\alpha)=f(\alpha),$ hence $\beta$ is the image of $\alpha$ via $g.$ Therefore, the antecedent of the formula in \cref{eq: surjection} holds, entailing a contradiction.
    
    The formula $\forall g(Fun(g, \hat{\gamma}),\forall x(x\rlzin \hat{\gamma}\imp \op(x, g(x))\rlzin f)\rightarrow\ Surj(g, \hat{\gamma}, \hat{\mu}))$ is equivalent to
    \[C(f, \hat{\mu}, \hat{\gamma})\equiv \forall g(Fun(g, \hat{\gamma}),\lnot Surj(g, \hat{\gamma}, \hat{\mu})\imp \exists x( x\rlzin \hat{\gamma}\land \op(x, g(x))\notrlzin f)),\]
    Thus, the formula in \cref{eq: surjection} can be rewritten as
    \begin{equation}\label{eq: pieta'}
        \forall f(Tot(f, \hat{\gamma}), C(f,\hat{\mu}, \hat{\gamma}) \imp \bot).
    \end{equation}
    We show that this formula is realized by a proof-like term that does not depend on $\mu$ nor $\gamma.$ 

    In order to define an explicit realizer for \cref{eq: pieta'}, we first need to fix some proof-like terms. 
    Observe that, as a consequence of Propositions \ref{prop:functionlift} and \ref{prop:domain of a function}, there exists a proof-like term $\theta_1$ such that for all $g:\gamma\rightarrow M^{\mathcal{K}}$, $\theta_1\Vdash Fun(\tilde{g},\hat{\gamma})$, and $\theta_1$ doesn't depend on $g$ nor $\gamma$. We also denote $\theta_2\dfn \lambda x.xI$, which is the quasi-proof given in Lemma \ref{lem:luglio}. Moreover, we can fix a proof-like term $\theta_3$ such that for every $\delta<\kappa$ and every formula $\varphi,$ we have $\theta_3\Vdash \exists x (x\rlzin \hat{\delta}\land \lnot \varphi(x))\imp \lnot \forall^{\hat{\delta}} x(x\rlzin \hat{\delta}\imp \varphi(x))$. This assumption relies on the fact that there exists a proof-like term which realizes the equivalence $\forall x^{\hat{\delta}}\varphi(x)\leftrightarrow \forall x (x\rlzin \hat{\delta}\imp \varphi(x))$, for every $\delta<\kappa$ and every formula $\varphi(x)$, independently of $\delta, \phi(x)$, together with the fact that the formula $\exists x(x\rlzin \hat{\delta}\land \lnot \varphi(x))$ is equivalent to $\lnot \forall x^{\hat{\delta}}(x\rlzin \hat{\delta}\imp \varphi(x))$.  
    
    With the notation just introduced, we show that  
    \[\lambda g\lambda l.(\BR g(\theta_3 (l\theta_1\theta_2)))\underline{0}\,\underline{0}\Vdash \forall f(Tot(f, \hat{\gamma}), C(f,\hat{\mu}, \hat{\gamma}) \imp \bot)\]
    by means of Theorem \ref{thm:BRtermination}.
    A stack in the falsity value of the formula is of the form $G\centerdot L\centerdot \pi$, where for some $f\in M^{\mathcal{K}}$, $G\in \falsitydual{Tot(f, \hat{\gamma})}$, $L\in \falsitydual{C(f, \hat{\gamma}, \hat{\mu})}$ and $\pi\in \Pi$. If we show that for some formula $F$, the terms $G$ and $U=\theta_3 (L \theta_1 \theta_2)$ meet the requirements of the theorem, then $\BR GU\underline{0}\,\underline{0}\Vdash\bot$, which entails in particular that $\BR GU\underline{0}\,\underline{0}\star\pi$ is in the pole, and so is $\lambda g\lambda l.(\BR g(\theta_3 (l\theta_1\theta_2)))\underline{0}\,\underline{0}\star G\centerdot L\centerdot\pi$.
    Consider the formula 
    \[F[f, x,y]\equiv x\rlzin \hat{\gamma}\imp \op(x, y)\rlzin f.\]
    Since $Tot(f, \hat{\gamma})$ corresponds to $\forall x\exists yF[f, x, y]$,  $G\in \falsitydual{\forall x \exists y F[f, x, y]}$, therefore the result follows from verifying that $\theta_3(L\theta_1\theta_2)\in\bigcap_{g:\gamma\rightarrow M^{\mathcal{K}}}\falsitydual{\neg\forall x^{\hat{\gamma}}(x\rlzin\hat{\gamma}\rightarrow\op(x,\tilde{g}(x))\rlzin f)}$.
    
    For every function $g:\gamma\rightarrow M^{\mathcal{K}}$, $\theta_1\Vdash Fun(\tilde{g}, \hat{\gamma})$ and $\theta_2\Vdash \neg Surj(\tilde{g}, \hat{\gamma}, \hat{\mu})$. This second fact holds since $g$ cannot be a surjection onto $\dom{\hat{\mu}}$, since $\mu$ is a cardinal in the ground model. 
    It follows that $L \theta_1 \theta_2\in\bigcap_{g:\gamma\to M^\mathcal{K}}\falsitydual{\exists x( x\rlzin \hat{\gamma}\land \op(x, \tilde{g}(x))\notrlzin f))}$. 
    By definition, $\theta_3$ realizes $(\exists x( x\rlzin \hat{\gamma}\land \op(x, g(x))\notrlzin f)))\imp (\neg\forall x^{\hat{\gamma}} F[f, x, \tilde{g}(x)]),$ thus for $U\dfn \theta_3 (L \theta_1 \theta_2)$ we have 
    \[U\in\bigcap_{g:\gamma\to M^\mathcal{K}}\falsitydual{\neg\forall x^{\hat{\gamma}} F[f, x, \tilde{g}(x)]}.\]
    This completes the proof, since $G$ and $U$ meet the requirements of \cref{thm:BRtermination}.
    %we have $\BR GU\underline{0}\,\underline{0}\Vdash{\bot},$ which entails $\BR G U \underline{0}\underline{0}\star \pi\in \Perp.$ From that, we can conclude that $\lambda g\lambda l.(\BR g(\theta_3 l\theta_1\theta_2))\underline{0}\,\underline{0} \star G\centerdot L\centerdot \pi\in \Perp.$
    \end{proof}

\section{An example of $\kappa$-\closed realizability algebra}\label{subsec: exemple}

In this section we build an example of realizability algebra that satisfies the $\kappa$-\closure property and the hypotheses of Theorem \ref{thm:BRtermination}. This is a generalization of the BBC algebra, which is a classical realizability algebra defined by Krivine in \cite{DBLP:conf/csl/Krivine16} and inspired by the work of Berardi, Bezem and Coquand in \cite{DBLP:journals/jsyml/BerardiBC98}. We generalize Krivine's construction to define for every regular cardinal $\kappa$, a (non-trivial) realizability algebra $\mathcal{K}=\langle\Lambda_\mathsf{c}, \Pi, \succ, \Perp\rangle$ that satisfies $\kappa$-\closure. 

We define the set of constants $A=\{\chi,\mathsf{h}\}\cup\{\mathsf{d}_\xi\mid\xi<\kappa\}$ where $\chi$ will be a trichotomy operator for a sequence of ordinal terms yet to define, $\mathsf{h}$ will play the role of a \emph{halting instruction}, and each $\mathsf{d}_\xi$ is a “dummy” constant. 
Moreover, we need $\Lambda_\mathsf{c}$ to contain terms representing accumulators of transfinite sequences; to this end, we assume that the set of terms contains for every limit ordinal $\alpha<\kappa,$ for each sequence of terms $(u_\xi)_{\xi<\alpha}\subseteq\Lambda_\mathsf{c}$, a term that we will denote $\bigwedge_{\xi<\alpha} u_\xi;$
terms of this form are called \emph{oracles} in \cite{DBLP:conf/csl/Krivine16}, and \emph{infinite terms} in \cite{DBLP:journals/jsyml/BerardiBC98}.

A sequence of terms $(\underline{\alpha})_{\alpha<\kappa}$ is inductively defined as follows: 

\begin{enumerate}
\item $\underline{0}\dfn I = \lambda x.x$
    \item $\underline{\alpha+1}\dfn \chi\underline{\alpha}\underline{\alpha}I$
    \item when $\alpha$ is limit, 
    $\underline{\alpha} \dfn 
    \bigwedge_{\xi<\alpha} I.$
\end{enumerate}

We consider $\Pi_0$ to be a set of cardinality $\kappa$. 
The rules of reduction are expanded to include the following:  
\begin{itemize}
\item  for any process $p$ and every stack $\pi$, we impose $\mathsf{h}\star\pi\dsucc p$ if and only if $p=\mathsf{h}\star\pi$, and $\mathsf{d}_\xi\star\pi\dsucc p$ if and only if $p=\mathsf{d}_\xi\star\pi$, $\xi<\kappa$. 
\item $\chi$ is a trichotomy operator with respect to the sequence $(\underline{\alpha})_{\alpha<\kappa},$ namely
for any $\alpha,\beta<\kappa$ and any $t,u\in\Lambda_\mathsf{c}, \pi\in\Pi$ we have 
        \[   
           (\chi)\quad \chi\underline{\alpha}\,fg\underline{\beta}\star\pi\succ_{\chi}
                \begin{cases}
                    (f)\underline{\beta}\star\pi&\textrm{if } \beta<\alpha;\\
                    g\star\pi&\textrm{ if $\beta\geq \alpha$}
                \end{cases}
        \]

\item  the following rule called \emph{lim} holds
        \[(lim)\quad\left(\bigwedge_{\xi<\alpha}u_\xi\right)\underline{\beta}\star\pi\succ_{lim} u_\beta\star\pi 
        %\left(\bigwedge_{\xi<\alpha}u_\xi\right)\false\star\pi\succ_{lim} u_0\star\pi;
        \]
    
\end{itemize}

%A natural number $n$ is encoded in $\Lambda_\mathsf{c}$ as $\underline{n}\dfn \underline{s}^n\underline{0}$, 
%where $\underline{0}\dfn \lambda x.x$  (i.e. $\underline{0}=I$), 

%Let $\underline{s}\dfn \lambda nx. x\,\false\lambda y. y\false n$ and $\false\dfn \lambda xy.y$. 

%By means of oracles, we can build $\underline{\alpha}$ for any ordinal $\omega\leq\alpha<\kappa$ as follows: 
%\begin{enumerate}
%\item $\underline{0}\dfn I = \lambda x.x$
 %   \item $\underline{\alpha}= \underline{s}\underline{\beta}$, if $\alpha=\beta+1$ 
 %    \item when $\alpha$ is limit, 
    %\bigwedge_{\xi<\alpha} \underline{\xi+1}$
    %$\underline{\alpha} = \lambda x.(\bigwedge_{\xi<\alpha} \underline{\xi+1})xx$ 
%\end{enumerate}

%We assume that the set of terms contains, for each limit ordinal $\lambda<\kappa$ and for each sequence of terms $(u_\xi)_{\xi<\lambda}\subseteq\Lambda_\mathsf{c}$, 
%an \emph{accumulator} for the sequence $(u_\xi)_{\xi<\lambda}$ that we will denote $\bigwedge_{\xi<\lambda} u_\xi$ such that for any stack $\pi$ and any ordinal $\alpha<\lambda,$ the following rule of reduction holds (we call it \emph{lim}): 
%        \[(lim)\quad\left(\bigwedge_{\xi<\lambda}u_\xi\right)\underline{\alpha}\star\pi\succ_{lim} u_\alpha\star\pi, \quad 
        % \left(\bigwedge_{\xi<\lambda}u_\xi\right)\false\star\pi\succ_{lim} u_0\star\pi;\]

We will show that every $\underline{\alpha}$ is an accumulator of the sequence $(I)_{\beta<\alpha}$ and that $(\underline{\alpha})_{\alpha<\kappa}$ satisfies Definition \ref{def:numerals}.

The pole $\Perp$ is defined by 
    \[\Perp\dfn\{p\in\Lambda_\mathsf{c}\star\Pi\mid \exists\pi(p\dsucc\mathsf{h}\star\pi)\}\]
and the set of proof-like terms $\qproofs$ is the set of closed terms of $\Lambda_\mathsf{c}$ not containing stack-constants nor $\mathsf{h}$ in between their subterms. The pole is trivially closed by antireduction and it is coherent, since all $t\in\qproofs$ can't reduce to $\mathsf{h}$. This concludes the definition of the algebra. 

\begin{remark}
    Observe that the pole $\Perp$ defined above satisfies the additional property of being closed by reduction, namely for $p,q$ processes such that $p\dsucc q$, if $p\in\Perp$ then $q\in\Perp$. This follows from the fact that the reduction is deterministic and that $\mathsf{h}$ is an halting instruction. This property fails in general; for instance, see the pole used to define the “thread model” at the end of \cite{DBLP:journals/corr/abs-1007-0825}.
\end{remark}

\begin{lemma}\label{lemma: ordinal terms accumulate sequence of 0} Every $\underline{\alpha}$ accumulates the sequence $(\underline{I})_{\xi<\alpha},$ 
\end{lemma}

\begin{proof}
We prove this statement by induction on $\alpha.$
For every limit ordinal $\alpha,$ 
$\underline{\alpha}$ accumulates the sequence $(I)_{\xi<\alpha}$ by definition. 
For the successor case, let $\beta<\alpha,$ then 
$\underline{\alpha+1}\underline{\beta}=\chi \underline{\alpha}\underline{\alpha}I\underline{\beta}$ which reduces to $I$ if $\beta=\alpha,$ else it reduces to 
$\underline{\alpha}\underline{\beta}$ which by induction hypothesis reduces to $I.$ \end{proof}

We want to prove that $(\underline{\alpha})_{\alpha<\kappa}$ is a system of ordinal terms and that this algebra satisfies the $\kappa$-\closure with respect to such ordinal terms. For that, we are going to prove that oracles correspond to minimal accumulators; we need some preliminary lemmas.

We extend the notion of substitution to accumulators: for a variable $x$, a constant $a$, a stack bottom $\sigma$, a term $t$ and a stack $\rho$,
\[\left(\bigwedge_{\xi<\lambda} u_\xi\right)\{x:=t\}=\bigwedge_{\xi<\lambda} (u_\xi\{x:=t\}), \quad \left(\bigwedge_{\xi<\lambda} u_\xi\right)\{a:=t\}=\bigwedge_{\xi<\lambda} (u_\xi\{a:=t\}),\]
\[\left(\bigwedge_{\xi<\gamma} u_\xi\right)\{\sigma:=\pi\}=\bigwedge_{\xi<\gamma} (u_\xi\{\sigma:=\pi\}).\]

In order to show that $\mathcal{K}$ satisfies the $\kappa$-\closure property, we need some intermediate results involving the notion of length of a reduction.

\begin{definition}
    The \emph{length of the reduction} between two processes $p$ and $q$ is defined as
        \[
        \len(p,q)\dfn
        \left\{
        \begin{tabular}{ll}
            $0$& if $p=q$;\\[2pt]
            $\len(p',q)+1$&\parbox[t]{.45\textwidth}{if there exists $p'$ such that $p\succ p'\dsucc q$;}\\[2pt]
            $\omega$&\text{if $p\centernot\dsucc q$}.
        \end{tabular}
        \right.
        \]
\end{definition}

\begin{remark}\label{rk:def pole by len}
    Observe that for a process $p$, $p\in\Perp$ if and only if there exist a integer $n$ and a stack $\rho$ such that $\len(p,\mathsf{h}\star\rho)=n$.
\end{remark}

\begin{lemma}\label{lmm:substack}
    In $\mathcal{K}$, for all $p,q\in\Lambda_\mathsf{c}\star\Pi$, $\sigma\in\Pi_0$, $\pi\in\Pi$, if $p\dsucc q$, then $\len(p,q)=\len(p\{\sigma:=\pi\}, q\{\sigma:=\pi\})$.
\end{lemma}
\begin{proof}
    By induction on the length of the reduction between $p$ and $q$. If $\len(p,q)=0$, then $p=q$ and $p\{\sigma:=\pi\}=q\{\sigma:=\pi\}$. If $\len(p,q)=n+1$, then there exists a process $p'$ such that $p\succ p'$ and $\len(p',q)=n$, thus by inductive hypothesis $\len(p'\{\sigma:=\pi\},q\{\sigma:=\pi\})=n$. To conclude, it suffices to show that $\len(p\{\sigma:=\pi\}, p'\{\sigma:=\pi\})=1$. We proceed by case analysis on the reduction $p\succ p'$.

    $(grab)$ case. $p=\lambda x.t\star u\centerdot\rho$, $p'= t\{x:=u\}\star\rho$ and 
    \[p\{\sigma:=\pi\}=\lambda x.(t\{\sigma:=\pi\})\star u\{\sigma:=\pi\}\centerdot\rho\{\sigma:=\pi\}\succ t\{\sigma:=\pi\}\{x:=u\{\sigma:=\pi\}\}\star\rho\{\sigma:=\pi\}.\] Observe that $t\{x:=u\}\{\sigma:=\pi\} = t\{\sigma:=\pi\}\{x:=u\{\sigma:=\pi\}\}$ (it can be proven by induction on $t$), thus the right-hand side of the reduction is equal to $t\{x:=u\}\{\sigma:=\pi\}\star\rho\{\sigma:=\pi\}=p'\{\sigma:=\pi\}$; in particular $\len(p\{\sigma:=\pi\},p'\{\sigma:=\pi\})=1$.

    $(push)$ case. $p=tu\star\rho$, $p'=t\star u\centerdot\rho$ and $p\{\sigma:=\pi\}$ is readily seen to reduce to $p'\{\sigma:=\pi\}$ by one step of $push$ reduction.

    $(save)$ case. $p=\cc\star t\centerdot\rho$, $p'=t\star k_{\rho}\centerdot\rho$, and
    \[p\{\sigma:=\pi\}=\cc\star t\{\sigma:=\pi\}\centerdot\rho\{\sigma:=\pi\} \succ t\{\sigma:=\pi\}\star k_{\rho\{\sigma:=\pi\}}\centerdot \rho\{\sigma:=\pi\}=p'\{\sigma:=\pi\}.\]
    
    $(restore)$ case. $p=k_{\varsigma}\star t\centerdot\rho$, $p'=t\star\varsigma$ and
    \[p\{\sigma:=\pi\}=k_{\varsigma\{\sigma:=\pi\}}\star t\{\sigma:=\pi\}\centerdot\rho\{\sigma:=\pi\}\succ t\{\sigma:=\pi\}\star\varsigma\{\sigma:=\pi\}=p'\{\sigma:=\pi\}.\]

    $(\chi)$ and $(lim)$ cases are similar to the previous ones. 
\end{proof}

\begin{corollary}\label{cor:redlength}
    For all $t\in\Lambda_\mathsf{c}^{\textrm{closed}}$ and $\sigma\in\Pi_0$, if $\sigma\not\in\substacks{t}$ and there exists $\rho\in\Pi$ such that $t\star\sigma\dsucc\mathsf{h}\star\rho$, then for each $\pi\in\Pi$ there exists $\rho'\in\Pi$ such that $\len(t\star\pi, \mathsf{h}\star\rho')=\len(t\star\sigma, \mathsf{h}\star\rho)$.
\end{corollary}
\begin{proof}
    It suffices to consider the processes $p=t\star\sigma, q=\mathsf{h}\star\rho$ in the previous lemma. Then, the result follows by setting $\rho'=\rho\{\sigma:=\pi\}$.
\end{proof}

\begin{proposition}\label{prop: oracles are minimal accumulators}
For every sequence $(u_\xi)_{\xi<\gamma}$ with $\gamma<\kappa$ limit, $\bigwedge_{\xi<\gamma} u_\xi$ is a minimal accumulator.  
\end{proposition}

\begin{proof} 
    Fix a sequence of terms $(u_\xi)_{\xi<\gamma}$ for some limit ordinal $\gamma<\kappa$. 
    Clearly, $t=\bigwedge_{\xi<\gamma} u_\xi$ accumulates it. 
    To show that the algebra satisfies the second item of Definition \ref{def: kappa-closure} we need to show that $t$ is a \emph{minimal} accumulator, so we need to find an homogeneous bound $\delta$ such that, for all terms $w$ such that $w$ accumulates $(u_\xi)_{\xi<\delta}$, 
    for every term $v\in\Lambda_\mathsf{c},$ the term $v\{x:=w\}$ is contained in $\falsitydual{\bot}$ whenever $v\{x:=t\}$ is. 
    Consider then $v$ such that $v\{x:=t\}\in\falsitydual{\bot}$, which means that for any stack $\pi$, $v\{x:=t\}\star\pi\in\Perp$. 
    Set $p_0\dfn v\{x:=t\}\star\sigma$ for $\sigma\in\Pi_0$ and suppose $\sigma\not\in \substacks{v}\cup\substacks{t}$ (this is possible since all terms contain at most $\lambda<\kappa$ different stack bottoms in between their substerms). Then, $p_0\dsucc \mathsf{h}\star\rho_0$ for some $\rho_0\in\Pi$. We can decompose the reduction in steps where $t$ is in head position and steps where $t$ isn't, thus the reduction can be rewritten as
        $p_0\{x:= t\}\succ p^{(1)}\succ\dots\succ p_i\succ p^{(j)}\succ\dots\succ q$
    where the processes $p_i$ with subscript indexes have $t$ in head position hence they must be in the form $t\underline{\beta_i}\star \pi_i$ 
    for $1\leq i\leq n$, and on the contrary $t$ does not appear at head position in the processes $p^{(j)}$. 
    Define $\delta\dfn\max_{1\leq i\leq n} \beta_i + 1,$ then $\delta<\gamma$ since $\gamma$ is a limit ordinal. Consider a term $w$, which accumulates $(u_\xi)_{\xi<\delta}$. Then, the process $v\{x:=w\}\star\sigma$ reduces in the same way as $v\{x:=t\}\star\sigma$, thus there exists $\rho_0'$ such that $\len(v\{x:=t\}\star\sigma, \mathsf{h}\star\rho_0)=\len(v\{x:=w\}\star\sigma, \mathsf{h}\star\rho_0')$. By Corollary \ref{cor:redlength} for all $\pi$, there is some $\rho'$ such that $\len(v\{x:=w\}\star\pi, \mathsf{h}\star\rho')=\len(v\{x:=w\}\star\sigma, \mathsf{h}\star\rho_0')$ thus $v\{x:=w\}\in\falsitydual{\bot}$.
\end{proof}

From this proposition, two results follow directly. 

\begin{corollary} 
$(\underline{\alpha})_{\alpha<\kappa}$ is a system of ordinal terms. 
\end{corollary}

\begin{proof}
%We just need to show item $2$ of Definition \ref{def:numerals}. 
Let $\delta$ be a limit ordinal, then by Proposition \ref{prop: oracles are minimal accumulators}, $\underline{\delta}$ is a minimal accumulator of $(I)_{\xi<\delta}.$ So if $v\{x:= \underline{\delta}\}\in \falsitydual{\bot},$ there is $\alpha$ such that for every accumulator $w$ of $(I)_{\xi<\alpha+1},$ we have $v\{x:= w\}\in \falsitydual{\bot}.$ In particular, since $\underline{\alpha+1}$ is an accumulator of $(I)_{\xi<\alpha+1}$ (by Lemma \ref{lemma: ordinal terms accumulate sequence of 0}), we have $v\{x:= \underline{\alpha+1}\}\in \falsitydual{\bot}.$
\end{proof}

\begin{corollary}\label{prop:example}
    The realizability algebra $\mathcal{K}$ defined above satisfies the $\kappa$-\closure property. 
\end{corollary}

We can define an inductive system of terms, from which it will follow that the structure we defined satisfies the hypothesis of Theorem \ref{thm:BRtermination}. 

\begin{proposition}
For $\kappa\leq |\mathcal{K}|,$ there is an inductive system of terms.
\end{proposition}

\begin{proof}
Let $\gamma<\kappa$ be a limit ordinal, we want to associate to every $\leq \delta$-sequence of terms $s:= (u_\xi)_{\xi<\delta}$ with $\delta< \gamma,$ an accumulator $t_s$ satisfying the properties of Definition \ref{def: canonical system}. For $\delta = 0$ or a successor ordinal, the definition of $t_s$ is already fixed by Definition \ref{def: canonical system}:    
\begin{enumerate}
        \item if $s$ is the empty sequence, then $t_s$ is $\underline{0}$ (any term is an accumulator of the empty sequence)
        \item if $\delta=\alpha+1$, then  $t_s\dfn (\chi)\underline{\alpha} t_{s\restr \alpha} u_\alpha$ 
\end{enumerate}
When $\delta$ is a limit ordinal $<\gamma,$ we let $t_s\dfn \bigwedge_{\alpha<\delta}(t_{s\restr {\alpha+1}}\underline{\alpha}).$ We conclude by showing that in this last case, given $\alpha<\delta$, for each term $v$ and each stack $\pi$, if $v\{x:=\Restr[t_s,\underline{\alpha}]\}\star\pi\in\Perp$ then $v\{x:=\Restr[t_{s\restr\alpha},\underline{\alpha}]\}\star\pi\in\Perp$. 
For each term $v$ and each stack $\pi$, define the process $p_{v,\pi}=v\{x:=\mathsf{d}\}\star\pi$ where $\mathsf{d}$ is a dummy constant (different from $\mathsf{h}$) not appearing in $v, \pi$. We show by induction on $n$ that for all $v,\pi$, if there exists a stack $\rho$ such that $\len(p_{v,\pi}\{\mathsf{d}:=\Restr[t_s,\underline{\alpha}]\}, \mathsf{h}\star\rho)=n$, then $p_{v,\pi}\{\mathsf{d}:=\Restr[t_{s\restr\alpha},\underline{\alpha}]\}\in\Perp$. This entails that for all $v$ and for all $\pi$,  $v\{x:=\Restr[t_s,\underline{\alpha}]\}\star\pi\in\Perp$ implies $v\{x:=\Restr[t_{s\restr\alpha},\underline{\alpha}]\}\star\pi\in\Perp$ (by Remark \ref{rk:def pole by len}); in particular, $v\{x:=\Restr[t_s,\underline{\alpha}]\}\in\falsitydual{\bot}$ implies $v\{x:=\Restr[t_{s\restr\alpha},\underline{\alpha}]\}\in\falsitydual{\bot}$.

Fix $v, \pi$ such that the assumption holds. In the following, we omit the subscript $v,\pi$ in $p_{v,\pi}$, and we denote $R=\Restr[t_s,\underline{\alpha}], R_\alpha=\Restr[t_{s\restr\alpha},\underline{\alpha}]$.  If $n=0$, then $p\{\mathsf{d}:=R\}=v\{x:=R\}\star\pi=\mathsf{h}\star\rho$, which implies $v=\mathsf{h}$ and therefore $p\{\mathsf{d}:=R_\alpha\}=\mathsf{h}\{x:=R_\alpha\}\star\pi=\mathsf{h}\star\pi\in\Perp$. 
For the inductive step, suppose that there exists a process $q$ such that $p\{\mathsf{d}:=R\}\succ q\dsucc \mathsf{h}\star\rho$. We proceed by case analysis on the reduction $p\{\mathsf{d}:=R\}\star\pi\succ q$, in order to rewrite $q$ as $p'\{\mathsf{d}:=R\}$ for some process $p'$ and then apply the inductive hypothesis.

$(grab)$ case. $v\{x:=\mathsf{d}\}=\lambda y.(v'\{x:=\mathsf{d}\})$, $\pi=w\centerdot\pi'$ and $q=v'\{x:=R\}\{y:=w\}\star\pi'$. Consider the process $p'=v'\{x:=\mathsf{d}\}\{y:=w\}\star\pi'$, which satisfies $p\dsucc p'$, $p\{\mathsf{d}:=R_\alpha\}\dsucc p'\{\mathsf{d}:=R_\alpha\}$ and $q=p'\{\mathsf{d}:=R\}$ (by hypothesis on $\mathsf{d}$, the substitutions $\{y:=w\}$ and $\{\mathsf{d}:=R\}$ commute). Therefore by inductive hypothesis $p'\{\mathsf{d}:=R_\alpha\}\in\Perp$ and so does $p\{\mathsf{d}:=R_\alpha\}$.

$(push)$ case. $v\{x:=\mathsf{d}\}=(v_1\{x:=\mathsf{d}\})(v_2\{x:=\mathsf{d}\})$ and $q=v_1\{x:=R\}\star v_2\{x:=R\}\centerdot\pi$. Let $p'=v_1\{x:=\mathsf{d}\}\star v_2\{x:=\mathsf{d}\}\centerdot\pi$, which satisfies as in the above case $p\dsucc p'$, $p\{\mathsf{d}:=R_\alpha\}\dsucc p'\{\mathsf{d}:=R_\alpha\}$ and $q=p'\{\mathsf{d}:=R\}$, so we conclude similarly. 

$(lim)$ case. $v\{x:=\mathsf{d}\}=\left(\bigwedge_{\xi<\alpha'}(v_\xi\{x:=\mathsf{d}\})\right)\underline{\beta}$, $q=v_\beta\{x:=R\}\star\pi$. Define $p'=v_\beta\{x:=\mathsf{d}\}$ and conclude as in the previous cases.

$(save)$ and $(restore)$ cases. Since $v\{x:=\mathsf{d}\}=\cc$ or $v\{x:=\mathsf{d}\}= k_\mathsf{\varsigma}$, the reduction doesn't depend on the substitution $\{\mathsf{d}:=R\}$, hence the cases is trivial. 

$(\chi)$ case. There are two subcases, depending on the shape of $v$.
    \begin{enumerate}
        \item $v\{x:=\mathsf{d}\}=(\mathsf{d}(v'\{x:=\mathsf{d}\}))\underline{\beta}$ for some ordinal $\beta<\kappa$. There are two cases depending on $\beta$.
        \begin{enumerate}
            \item If $\beta<\alpha$, then $q=(t_s)\underline{\beta}\star\pi$, and $p\{\mathsf{d}:=R_\alpha\}\dsucc (t_{s\restr\alpha})\underline{\beta}\star\pi$. By definition, $(t_s)\underline{\beta}\star\pi$ and $(t_{s\restr\alpha})\underline{\beta}\star\pi$ reduce to the same process $u_\beta\star\pi$. Using that the first process is in $\Perp$ and that the pole is closed by reduction, we deduce that $u_\beta\star\pi\in\Perp$, which entails the conclusion.
            \item If $\beta\geq\alpha$, define $p'=v'\{x:=\mathsf{d}\}\star\pi$. Observe that both $p\{\mathsf{d}:=R\}\succ p'\{\mathsf{d}:=R\}=q$ and $p\{\mathsf{d}:=R_\alpha\}\succ p'\{\mathsf{d}:=R_\alpha\}$, therefore by induction hypothesis we conclude. 
        \end{enumerate}
        \item $v\{x:=\mathsf{d}\}=((\chi\underline{\alpha'}(v_1\{x:=\mathsf{d}\}))(v_2\{x:=\mathsf{d}\}))\underline{\beta}$ for some $\alpha',\beta<\kappa$. Define $p'$ to be the process $(v_1\{x:=\mathsf{d}\})\underline{\beta}\star\pi$ if $\beta<\alpha'$, otherwise $p'=v_2\{x:=\mathsf{d}\}\star\pi$. In both cases, $p\{\mathsf{d}:=R\}\dsucc p'\{\mathsf{d}:=R\}=q$ and $p\{\mathsf{d}:=R_\alpha\}\dsucc p'\{\mathsf{d}:=R_\alpha\}$, therefore we conclude by inductive hypothesis.\qedhere
    \end{enumerate}
\end{proof}

\begin{comment} 
Let us write $u\sim v$ when for every $\pi\in \Pi,$ we have $u\star \pi\in \Perp\iff v\star \pi\in \Perp.$\\        
{\bf Claim:} for every $\beta<\alpha\leq \delta,$ $t_{s\restr \alpha}\underline{\beta}\sim u_\beta.$\\

\emph{Proof of claim.} We prove the claim by induction on $\alpha.$ 
For the successor case, we have 
$t_{s\restr{\alpha+1}}\underline{\beta}=\chi \underline{\alpha}t_{s\restr \alpha}u_\alpha \underline{\beta}\sim t_{s\restr \alpha}\underline{\beta}$ which by inductive hypothesis is $\sim u_\beta.$
When $\alpha$ is limit, we have $t_{s\restr \alpha}\underline{\beta}= (\bigwedge (t_{s\restr {\xi+1}}\underline{\xi}))\underline{\beta}\sim t_{s\restr \beta+1}\underline{\beta}$ which by inductive hypothesis is $\sim u_\beta.$ \emph{End proof of claim.}
        
It follow from the claim that for every $\beta<\alpha,$ we have $\Restr[\underline{\beta}, t_{s\restr \beta}]\sim \Restr[\underline{\beta}, t_{s\restr \alpha}],$ in particular for $\alpha=\delta$ limit. From this, it follows that item 3. of Definition \ref{def: canonical system} is satisfied. \todo{SCRIVERE MEGLIO LA PROVA}
\end{comment}

We want to show that the exhibited algebra is a genuine example of non trivial $\kappa$-\closed realizability algebra, as it can be realized that $\fullname{2}$ contains more than two non-extensional elements. 

\begin{lemma}\label{lmm:subconstant}
    In $\mathcal{K}$, for all $p,q\in\Lambda_\mathsf{c}\star\Pi$, $a\in\{\mathsf{h}\}\cup\{\mathsf{d}_\xi\mid\xi<\kappa\}$, $t\in\Lambda_\mathsf{c}^{\textrm{closed}}$, if $p\dsucc q$, then $\len(p,q)=\len(p\{a:=t\}, q\{a:=t\})$.
\end{lemma}
\begin{proof}
    The proof is similar to the one of Lemma \ref{lmm:substack}, we omit it.
\end{proof}

The following is a analogue of the usual substitution lemma. It states the substitution of any two constants of the form $\mathsf{h}, \mathsf{d}_\xi$ commutes with the reduction.

\begin{lemma}\label{lmm:antisub}
    Let $a,b\in \{\mathsf{h}\}\cup\{\mathsf{d}_\xi\mid\xi<\kappa\}, a\neq b$ and suppose $p\{a:=b\}\succ q$. Then there exists $p'$ such that $p\succ p'$ and $p'\{a:=b\}=q$.
\end{lemma}
\begin{proof}
    The proof is made by case analysis on the reduction $p\{a:=b\}\succ q$. Let $p=t\star\pi$. 

    $(grab)$ case. $t\{a:=b\}=\lambda x.t'$, $\pi\{a:=b\}=u\centerdot\pi'$ and $q=t'\{x:=u\}\star\pi'$. By definition of substitution, we must have $t=\lambda x.t''$, $\pi=u'\centerdot\pi''$ and $t'= t''\{a:=b\}$, $u= u'\{a:=b\}, \pi'=\pi''\{a:=b\}$. Define $p'=t''\{x:=u'\}\star\pi''$. Evidently $p\succ p'$, and 
    \begin{align*}
        p'\{a:=b\} &= t''\{x:=u'\}\{a:=b\}\star\pi''\{a:=b\}\\
        &=t''\{a:=b\}\{x:=u'\{a:=b\}\}\star\pi''\{a:=b\}\\
        &=t'\{x:=u\}\star\pi' = q.
    \end{align*}

    $(push)$ case. $t\{a:=b\}=(t_1)t_2, \pi\{a:=b\}=\pi', q=t_1\star t_2\centerdot\pi$. There must exist $t_1', t_2'$ such that $t_1=t_1'\{a:=b\}, t_2=t_2'\{a:=b\}$ and $t=(t_1'')t_2''$. Then, $p'=t_1'\star t_2'\centerdot\pi$.

    $(save)$ case. $t\{a:=b\}=\cc, \pi\{a:=b\}=u\centerdot\pi'$ and $q=u\star k_\pi'\centerdot\pi'$. Then, $t=\cc$ and there exist $u',\pi''$ such that $u=u'\{a:=b\}, \pi'=\pi''\{a:=b\}$. Set $p'=u'\star k_{\pi''}\centerdot\pi''$, which satisfies $p\succ p'$ and $p'\{a:=b\}=q$.

    $(restore)$ case. $t\{a:=b\} = k_\rho, \pi\{a:=b\} = u\centerdot\pi', q=u\star\rho$. Then, there exist $u',\rho',\pi''$ such that $u=u'\{a:=b\}, \pi'=\pi''\{a:=b\}, \rho=\rho'\{a:=b\}$. Therefore, $t=k_{\rho'}, \pi=u'\centerdot \pi''$ and $p'=u'\star\rho'$ verifies the sought conditions.

    $(\chi)$ case. $t\{a:=b\}=\chi\underline{\alpha}\,fg\underline{\beta}$ for some $\alpha,\beta<\kappa$, $\pi\{a:=b\}=\pi'$, and $q$ is either equal to $f\underline{\alpha}\star\pi'$ or $g\star\pi'$. Observe that for all $a,b$ as in the hypothesis, $\chi\{a:=b\}=\chi$ and $\underline{\xi}\{a:=b\}=\underline{\xi}$ for all $\xi<\kappa$. Therefore, $t=\chi\underline{\alpha}\,f'g'\underline{\beta}$ for some 
    terms $f',g'$ such that $f'\{a:=b\}=f, g'\{a:=b\}=g$. Thus, $p=\chi\underline{\alpha}\,f'g'\underline{\beta}\star\pi$, which reduces to $p'$ such that $p'\{a:=b\}=q.$
    
    $(lim)$ case. $t\{a:=b\}=(\bigwedge_{\xi<\gamma} u_\xi)t'$ for $t'=\underline{\alpha}$, $\pi\{a:=b\}=\pi'$, and $q=u_\beta\star\pi'$. There must exists a sequence $(u'_\xi)_{\xi<\gamma}$ such that for all $\xi<\gamma$, $u'_\xi\{a:=b\} = u_\xi$. Then, $p=(\bigwedge_{\xi<\gamma} u'_\xi)t'\star\pi$ reduces to $p'=u'_\beta\star\pi$, and $p'\{a:=b\}=q$.
\end{proof}

From the previous lemma it follows that, given a process $p$ and a dummy constant $a$ such that $p\{a:=\mathsf{h}\}\dsucc \mathsf{h}\star\rho$ for some stack $\rho$, then there exists $p'$ such that $p\dsucc p'$ and $p'\{a:=\mathsf{h}\}=\mathsf{h}\star\rho$. Observe that a priori the substitution $\{a:=\mathsf{h}\}$ may not be relevant for reducing in a process of the form $\mathsf{h}\star\rho$: take for instance $p=\lambda x.x\star\mathsf{h}\centerdot a\centerdot\pi$, for some stack $\pi$. In this case, we have that $p\{a:=u\}\in\Perp$ \emph{for all $u\in\Lambda_\mathsf{c}^{\textrm{closed}}$}. Next result determines when this is the case. 

\begin{lemma}\label{lmm:transparence}
    Let $a,b\in \{\mathsf{h}\}\cup\{\mathsf{d}_\xi\mid\xi<\kappa\}, a\neq b$ and suppose $p\{a:=b\}\dsucc b\star\rho$ for some $\rho$. Then for all $u\in\Lambda_\mathsf{c}^{\textrm{closed}}$, $u\neq b$, there exists a stack $\varsigma$ such that one of the following hold:
    \begin{itemize}
        \item $\len(p\{a:=u\}, b\star\varsigma)=\len(p\{a:=b\}, b\star\rho)$, and for any other $u'\in\Lambda_\mathsf{c}^{\textrm{closed}}$, there exists $\varsigma'$ such that $\len(p\{a:=u'\}, b\star\varsigma')=\len(p\{a:=b\}, b\star\rho)$; or
        \item $\len(p\{a:=u\}, u\star\rho')=\len(p\{a:=b\}, u\star\rho)$ and for any other $u'\in\Lambda_\mathsf{c}^{\textrm{closed}}$, there exists $\varsigma'$ such that $\len(p\{a:=u'\}, u'\star\varsigma')=\len(p\{a:=b\}, b\star\rho)$
    \end{itemize}
\end{lemma}
\begin{proof}
    Let $p=t\star\pi$ such that $p\{a:=b\}=t\{a:=b\}\star(\pi\{a:=b\})\dsucc b\star\rho$. We show it by induction on $\len(p\{a:=b\}, u\star\rho)$. If $\len(p\{a:=b\}, b\star\rho)=0$, then $t\{a:=b\}=b, \pi\{a:=b\}=\rho$. We distinguish two cases.
    \begin{itemize}
        \item If $t=b$, then $t\{a:=u\}=b$ and for $\varsigma=\pi\{a:=u\}$ we have $p\{a:=u\}= b\star\varsigma$. If we consider another $u'$, then for $\varsigma'=\pi\{a:=u'\}, p\{a:=u'\}=b\star\varsigma'$.
        \item If $t=a$, then $t\{a:=u\}=u$ and for $\varsigma=\pi\{a:=u\}$ we have $p\{a:=u\}= u\star\varsigma$. For any other $u'$, $p\{a:=u'\}=u'\star\varsigma'$ for $\varsigma'=\pi\{a:=u'\}$.
    \end{itemize}
    Suppose that $\len(p\{a:=b\}, u\star\rho)=n+1$. Then, by Lemma \ref{lmm:antisub} there exists $p'$ such that $p\succ p'$ and $\len(p'\{a:=b\}, b\star\rho)=n$. By inductive hypothesis, we have two cases:
    \begin{itemize}
        \item for some $\varsigma$, $\len(p'\{a:=u\}, b\star\varsigma)=n$ and for all $u'$ there exists $\varsigma'$ such that $\len(p'\{a:=u'\}, b\star\varsigma')=n$. By Lemma \ref{lmm:subconstant}, $\len(p\{a:=u\}, p'\{a:=u\})=\len(p,p')=1$, therefore $\len(p\{a:=u\}, b\star\varsigma)=n$. The same holds for $u'$.
        \item for some $\varsigma$, $\len(p'\{a:=u\}, u\star\varsigma)=n$ and for all $u'$ there exists $\varsigma'$ such that $\len(p'\{a:=u'\}, u'\star\varsigma')=n$. By Lemma \ref{lmm:subconstant}, $\len(p\{a:=u\}, p'\{a:=u\})=\len(p,p')=1$, therefore $\len(p\{a:=u\}, u\star\varsigma)=n$, and similarly for $u'$. \qedhere
    \end{itemize}
\end{proof}

\begin{lemma}
    If $t\in\falsitydual{\bot,\top\rightarrow\bot}\cap\falsitydual{\top,\bot\rightarrow\bot}$, then $t\in\falsitydual{\top,\top\rightarrow\bot}$.
\end{lemma}
\begin{proof}
Suppose by contradiction that for some $u, v, \pi$ the process $q=t\star u\centerdot v\centerdot\pi\not\in\Perp$; notice that $u$ and $v$ are different from $\mathsf{h}$. Consider $d,e\in\{\mathsf{d}_\xi\mid \xi<\kappa\}$ such that $d\neq e, d,e\not\in\subterm{t}\cup\subterm{u}\cup\subterm{v}\cup\subterm{\pi}$. Let $p=t\star d\centerdot e\centerdot\pi$, $p_1=p\{d:=\mathsf{h}\}, p_2=p\{e:=\mathsf{h}\}$. By hypothesis on $t$, there exist two stacks $\rho_1, \rho_2$ and two integers $n_1, n_2$ such that 
\[\len(p_1,\mathsf{h}\star\rho_1)=n_1,\quad \len(p_2, \mathsf{h}\star\rho_2)=n_2.\]
Then, by Lemma \ref{lmm:subconstant}, for all terms $u,v$,
\[\len(p_1\{e:=v\},\mathsf{h}\star\rho_1\{e:=v\})=n_1,\quad \len(p_2\{d:=e\}, \mathsf{h}\star\rho_2\{d:=e\})=n_2.\]
We claim that $n_1=n_2$. Take $u=v=\mathsf{h}$ in previous equations. Since $p_1\{e:=\mathsf{h}\}=p_2\{d:=\mathsf{h}\}$ and the reduction is deterministic, we have that $\len(\mathsf{h}\star\rho_1\{e:=v\}, \mathsf{h}\star\rho_2\{d:=e\})=n_2-n_1$. Since $\mathsf{h}\star\rho_1\{e:=v\}=\mathsf{h}\star\rho_2\{d:=e\}$ because $\mathsf{h}$ has no reduction rules, $n_2-n_1=0$. Let $n=n_1=n_2$. Define now $p_1'=p\{d:=u\}$ and $p_2'=p\{e:=v\}$. By Lemma \ref{lmm:transparence}, there exists $p_1''$ such that $\len(p_1',p_1'')=n$ and one of the two conditions of the lemma apply. The same holds for $p_2'$. Observe now that $q=p_1'\{e:=v\}=p_2'\{d:=v\}$ and again by Lemma \ref{lmm:subconstant}
\[\len(q, p_1''\{e:=v\})=n=\len(q, p_2''\{d:=u\}),\]
which entails that $p_1''\{e:=v\}=p_2''\{d:=u\}$. The processes $p_1''$ and $p_2''$ cannot present $\mathsf{h}$ in head position, otherwise $p_1''\{e:=v\}$ and $p_2''\{d:=u\}$ would present it too ($e$ and $d$ are different from $\mathsf{h}$) and $q$ would be in $\Perp$. Therefore, 
$p''_1=u\star\varsigma_1, p''_2=v\star\varsigma_2$ for some stacks $\varsigma_1, \varsigma_2$. Following Lemma \ref{lmm:transparence}, this implies that, for all $u', v'$ there exist $\varsigma_1', \varsigma_2'$ such that $\len(p\{d:=u'\}, u'\star\varsigma_1')=n, \len(p\{e:=v'\}, v'\star\varsigma_2')=n$. Let $\varsigma_K,\varsigma_F$ be the stacks such that $\len(p\{d:=\lambda xy.x\}, \lambda xy.x\star\varsigma_K)=n, \len(p\{e:=\lambda xy.y\}, \lambda xy.y\star\varsigma_F)=n$. Set $q'=p\{d:=\lambda xy.x\}\{e:=\lambda xy.y\}=p\{e:=\lambda xy.y\}\{d:=\lambda xy.x\}$. Then we have
\[\len(q', \lambda xy.x\star\varsigma_K\{e:=\lambda xy.y\})=n=\len(q', \lambda xy.y\star\varsigma_F\{d:=\lambda xy.x\}),\]
therefore $\lambda xy.x\star\varsigma_K\{e:=\lambda xy.y\})=\lambda xy.y\star\varsigma_F\{d:=\lambda xy.x\})$, which is a contradiction.
\end{proof}

\begin{proposition}
    In the realizability model generated by $\mathcal{K}$, the ordinal $\fullname{2}$ has more than two non-extensional elements.
\end{proposition}
\begin{proof}
    To show that $\Vdash \neg\forall x^{\fullname{2}} (x\neq\fullname{0}, x\neq\fullname{1}\rightarrow \bot)$, consider $t\in\falsitydual{\forall x^{\fullname{2}} (x\neq\fullname{0}, x\neq\fullname{1}\rightarrow \bot)}$, $\pi\in\Pi$. From the previous lemma, $t\in\falsitydual{\bot,\top\rightarrow\bot}\cap\falsitydual{\top,\bot\rightarrow\bot}\subseteq \falsitydual{\top,\top\rightarrow\bot}$, therefore $(t)uv\in\falsitydual{\bot}$ for any $u, v$, in particular for $u=v=I$. Hence, $\lambda x.(x)II$ realizes that $\fullname{2}$ is not trivial.
\end{proof}

\section{Conclusions}

In this paper, we introduced a transfinite form of bar-induction and the corresponding transfinite bar-recursion operator in the setting of classical realizability. Combined with the notion of $\kappa$-\closure, this construction allowed us to obtain realizability models validating fragments of the Axiom of Choice indexed below $\hat{\kappa}$ while preserving cardinals up to $\hat{\kappa}$. These results provide further evidence that bar-recursion remains closely connected to choice principles beyond the countable setting and suggest a fruitful interaction between classical realizability, transfinite recursion, and techniques originating from forcing theory.

Several questions remain open. The methods developed here appear sufficiently flexible to be adapted to other choice principles. In particular, we believe that a suitable refinement of our construction could be used to realize $DC_{<\hat{\kappa}}$ or even Zorn's Lemma restricted to $\gamma$ for every limit ordinal $\gamma$ below $\hat{\kappa}$. Establishing such a result would require additional work and will be the subject of future investigations.

Moreover, since the present work is a  generalization of a realizability model used by Krivine to realize the Continuum Hypothesis, 
the overall architecture of our construction suggests that analogous techniques could be employed to obtain realizability interpretations of instances of the Generalized Continuum Hypothesis below $\hat{\kappa}$. 

We leave these developments for future work. More generally, we hope that the present work contributes to a better understanding of the computational content of uncountable mathematical principles and of the role played by ordinals and transfinite recursion in classical realizability models of set theory.

\bibliographystyle{plain}
\bibliography{closure}

\end{document}